\documentclass[11pt,a4paper]{amsart}
\pdfoutput=1
\usepackage{amsmath,amsfonts,amssymb,amsthm,epsfig,amscd,psfrag,latexsym,
comment}
\usepackage{caption}
\usepackage{graphicx}
\usepackage{array}
\usepackage{subfigure}
\usepackage[all]{xy}
\input xy
\xyoption{all}
\usepackage[colorlinks=true, pdfpagemode=none, pdfmenubar=false, pdfstartview=FitH, linkcolor=blue, citecolor=blue, urlcolor=blue, pdffitwindow=false]{hyperref}

\title{Reid's recipe and derived categories}
\author{Timothy Logvinenko} 
\email{T.Logvinenko@liv.ac.uk} 
\address{Department of Mathematical Sciences\\
University of Liverpool\\ 
Peach Street\\ 
Liverpool, L69 7ZL\\
UK}

\DeclareMathOperator{\eend}{End}

\DeclareMathOperator{\gl}{GL}
\DeclareMathOperator{\gsl}{SL}

\DeclareMathOperator{\Supp}{Supp}
\DeclareMathOperator{\picr}{Pic}

\DeclareMathOperator{\cl}{Cl}

\DeclareMathOperator{\hilb}{Hilb}

\DeclareMathOperator{\irr}{Irr}

\DeclareMathOperator{\supp}{Supp}

\DeclareMathOperator{\cohcat}{Coh}

\DeclareMathOperator{\lder}{\bf L}
\DeclareMathOperator{\rder}{\bf R}

\DeclareMathOperator{\hex}{Hex}
\DeclareMathOperator{\Except}{Exc}

\begin{document}

\def\bv{\mathbf{v}}
\def\kgc_{K^*_G(\mathbb{C}^n)}
\def\kgchi_{K^*_\chi(\mathbb{C}^n)}
\def\kgcf_{K_G(\mathbb{C}^n)}
\def\kgchif_{K_\chi(\mathbb{C}^n)}
\def\gpic_{G\text{-}\picr}
\def\gcl_{G\text{-}\cl}
\def\trch_{{\chi_{0}}}
\def\regring{{R}}
\def\regrep{{V_{\text{reg}}}}
\def\givrep{{V_{\text{giv}}}}
\def\lbar{{(\mathbb{Z}^n)^\vee}}
\def\genpx_{{p_X}}
\def\genpy_{{p_Y}}
\def\genpcn_{p_{\mathbb{C}^n}}
\def\gnat{gnat}
\def\twalg{{\regring \rtimes G}}
\def\L{{\mathcal{L}}}
\def\O{{\mathcal{O}}}
\def\gcd{\mbox{gcd}}
\def\lcm{\mbox{lcm}}
\def\tf{{\tilde{f}}}
\def\tD{{\tilde{D}}}

\def\mckquiv{\mbox{Q}(G)}
\def\C{{\mathbb{C}}}
\def\sF{{\mathcal{F}}}
\def\sW{{\mathcal{W}}}
\def\sL{{\mathcal{L}}}
\def\O{{\mathcal{O}}}
\def\Z{{\mathbb{Z}}}
\def\hmone{{\mathcal{W}}}

\theoremstyle{definition}
\newtheorem{defn}{Definition}[section]
\newtheorem*{defn*}{Definition}
\newtheorem{exmpl}[defn]{Example}
\newtheorem*{exmpl*}{Example}
\newtheorem{exrc}[defn]{Exercise}
\newtheorem*{exrc*}{Exercise}
\newtheorem*{chk*}{Check}
\newtheorem*{remarks*}{Remarks}
\theoremstyle{plain}
\newtheorem{theorem}{Theorem}[section]
\newtheorem*{theorem*}{Theorem}
\newtheorem{conj}[defn]{Conjecture}
\newtheorem*{conj*}{Conjecture}
\newtheorem{prps}[defn]{Proposition}
\newtheorem*{prps*}{Proposition}
\newtheorem{cor}[defn]{Corollary}
\newtheorem*{cor*}{Corollary}
\newtheorem{lemma}[defn]{Lemma}
\newtheorem*{claim*}{Claim}
\newtheorem{Specialthm}{Theorem}
\renewcommand\theSpecialthm{\Alph{Specialthm}}
\numberwithin{equation}{section}
\renewcommand{\textfraction}{0.001}
\renewcommand{\topfraction}{0.999}
\renewcommand{\bottomfraction}{0.999}
\renewcommand{\floatpagefraction}{0.9}
\setlength{\textfloatsep}{5pt}
\setlength{\floatsep}{0pt}
\setlength{\abovecaptionskip}{2pt}
\setlength{\belowcaptionskip}{2pt}
\begin{abstract}
We prove two conjectures from 
\cite{CautisLogvinenko-ADerivedApproachToGeometricMcKayCorrespondence}
which describe the geometrical McKay correspondence for a finite abelian 
$G \subset \gsl_3(\mathbb{C})$ such that $\mathbb{C}^3/G$ has a single
isolated singularity. We do it by studying the relation between 
the derived category mechanics of computing a certain Fourier-Mukai
transform and a piece of toric combinatorics known as `Reid's
recipe', effectively providing a categorification of the latter.  
\end{abstract}

\maketitle

\section{Introduction} \label{section-intro}

The classical McKay correspondence is a one-to-one correspondence 
\begin{align*}
\xymatrix{
\irr(G) \setminus \rho_0 \quad \ar@{<->}[rr]^{\;\;1\text{-to-}1}
& & \quad \Except(Y)
}
\end{align*}
between the non-trivial irreducible representations of a finite subgroup 
$G$ of $\gsl_2(\mathbb{C})$ and the irreducible exceptional divisors on 
the minimal resolution $Y$ 
of the singular quotient space $\mathbb{C}^2/G$. It first arose from 
an observation by McKay in \cite{McKay-GraphsSingularitiesAndFiniteGroups} 
which implied a coincidence of the representation graph of $G$, less the 
trivial representation $\rho_0$, and the intersection graph of $\Except(Y)$. 
Gonzales-Sprinberg and Verdier in 
\cite{GsV-ConstructionGeometriqueDeLaCorrespondanceDeMcKay} gave a
geometric construction where this coincidence was shown to arise 
naturally from a $K$-theory isomorphism 
$\Theta:\; K^G(\mathbb{C}^3) \rightarrow K(Y)$ between 
the $G$-equivariant $K$-theory of $\mathbb{C}^2$ and 
the $K$-theory of $Y$. In modern language, $\Theta$ is defined by
identifying $Y$ with $G$-$\hilb(\mathbb{C}^2)$, the fine moduli 
space of $G$-clusters\footnote{A \em $G$-cluster \rm is a
finite-length $G$-invariant subscheme $Z$ such that $H^0(Z)$ 
is isomorphic to the regular representation of $G$. It serves as 
a scheme-theoretic generalization of a concept of a set-theoretic
orbit of $G$.} in $\mathbb{C}^2$, and setting $\Theta$
to be the $K$-theoretic Fourier-Mukai transform defined by 
the universal $G$-cluster family $\mathcal{M}$ on 
$Y \times \mathbb{C}^2$
\begin{align*}
\Theta(-) = \left[\pi_{Y *} \left( \mathcal{M} \otimes
\pi^*_{\mathbb{C}^2} (-) \right) \right]^G
\end{align*}
where $\pi_{Y}$ and $\pi_{\mathbb{C}^2}$ are projections from 
$Y \times \mathbb{C}^2$ to $Y$ and $\mathbb{C}^2$, respectively.
The functor $\left[-\right]^G: K^G(Y) \rightarrow K(Y)$ is
the functor of taking the $G$-invariant part of a $G$-sheaf. 
It was then proved in
\cite{GsV-ConstructionGeometriqueDeLaCorrespondanceDeMcKay}
that for every $\rho \in \irr(G) \setminus \rho_0$ there exists
a unique $E_\rho \in \Except(Y)$ such that $\Theta(\mathcal{O}_0 \otimes 
\rho) = \left[ \mathcal{O}_{E_\rho}(-1) \right]$, where $\mathcal{O}_0$
is the skyscraper sheaf of the origin $(0,0) \in \mathbb{C}^2$. 
The group $G$ acts on $Y$ trivially, so every $G$-sheaf $\mathcal{F}$
on $Y$ splits up as a direct sum 
$\bigoplus_{\rho \in \irr(G)} \mathcal{F}_\rho \otimes \rho$ where
each $\mathcal{F}_\rho$ is a $G$-invariant sheaf called \em the
$\rho$-eigensheaf \rm of $\mathcal{F}$. Observe that not only we have
$\left[\mathcal{F}\right]^G = \mathcal{F}_{\rho_0}$, by definition, but 
more generally 
$\left[\mathcal{F} \otimes \rho\right]^G = \mathcal{F}_{\rho^\vee}$
for every $\rho \in \irr(G)$. Thus by looking at
$\Theta(\mathcal{O}_0 \otimes \rho)$ we are looking at how does
$\pi_{Y *}\left(\mathcal{M}_{|{\Except(Y) \times \{0\}}} \right)$
break up into $G$-eigensheaves. Very roughly, to obtain the
correspondence $\irr(G) \setminus \rho_0 \leftrightarrow \Except(Y)$ 
we break up the exceptional set of $Y$ with respect to the $G$-action
on its natural $G$-cluster scheme structure and observe that 
for each non-trivial $\rho \in \irr(G)$ we get a different irreducible
curve.  

In
\cite{CautisLogvinenko-ADerivedApproachToGeometricMcKayCorrespondence}
we proposed a program of the geometric McKay correspondence
which generalises the ideas of 
\cite{GsV-ConstructionGeometriqueDeLaCorrespondanceDeMcKay} to 
dimension three. In a celebrated result of \cite{BKR01} it was shown 
that the $K$-theoretic equivalence $\Theta$ of 
\cite{GsV-ConstructionGeometriqueDeLaCorrespondanceDeMcKay} lifts
naturally to the level of derived categories and gives for any 
finite subgroup $G \subset \gsl_{n}(\mathbb{C})$, where $n = 2$ or $3$, 
the equivalence $\Phi: D(Y) \rightarrow D^G(\mathbb{C}^n)$ between 
the bounded derived categories of coherent sheaves on the distinguished 
crepant resolution $Y = G$-$\hilb(\mathbb{C}^{n})$ and of coherent
$G$-sheaves on $\mathbb{C}^{n}$. In
\cite{CautisLogvinenko-ADerivedApproachToGeometricMcKayCorrespondence}
we showed that the inverse $\Psi: D^G(\mathbb{C}^{n}) \rightarrow D(Y)$
of $\Phi$ is the Fourier-Mukai transform
\begin{align*}
\Psi(-) = \left[\pi_{Y *} \left( \tilde{\mathcal{M}}
\overset{\lder}{\otimes} \pi^*_{\mathbb{C}^n} (-) \right) \right]^G
\end{align*}
defined by the dual family $\tilde{\mathcal{M}}$ of the universal family 
$\mathcal{M}$ of $G$-clusters on $Y \times \mathbb{C}^n$. We then 
showed how to compute the transforms $\Psi(\mathcal{O}_0 \otimes \rho)$ 
and although apriori each of these transforms is a
complex in $D(Y)$ we were able to show that for an abelian $G$ all
the cohomologies of this complex vanish except for one, i.e. for every 
$\rho \in \irr(G)$ the transform $\Psi(\mathcal{O}_0 \otimes \rho)$ is 
a shift of a coherent sheaf
(\cite{CautisLogvinenko-ADerivedApproachToGeometricMcKayCorrespondence}, 
Theorem 1.1). This is expected to also hold for non-abelian $G$. 
We then proposed the geometric McKay correspondence 
to be $\rho \mapsto \supp \left(\Psi(\mathcal{O}_0 \otimes
\rho)\right)$, assigning to every $\rho$ a closed subscheme of 
the exceptional set of $Y$. In dimension $2$ this gives precisely 
the classical $\irr(G) \setminus \rho_0 \leftrightarrow \Except(Y)$
correspondence of \cite{GsV-ConstructionGeometriqueDeLaCorrespondanceDeMcKay}. 
In dimension $3$ the correspondence is more complicated - as was expected
given that generally $G$ has more irreducible representations then there 
are irreducible divisors on $Y$. 
For abelian $G \subset \gsl_3(\mathbb{C})$ we were able 
to employ the numerical methods of 
\cite{Logvinenko-Natural-G-Constellation-Families} 
to compute this correspondence very explicitly
(\cite{CautisLogvinenko-ADerivedApproachToGeometricMcKayCorrespondence},
\S6). Based on numerous computational evidence we made a 
conjecture as to the form that $\Psi(\mathcal{O}_0 \otimes \rho)$
take in dimension three and in the present paper we prove it:
\begin{theorem}(\cite{CautisLogvinenko-ADerivedApproachToGeometricMcKayCorrespondence},
Conj. 1.1) \label{theorem-geometrical-mckay-correspondence}
Let $G \subset \gsl_3(\mathbb{C})$ be a finite abelian subgroup such 
that $\mathbb{C}^3/G$ has a single isolated singularity at the origin. 
Then for any $\chi \in \irr(G)$ the Fourier-Mukai transform 
$\Psi(\mathcal{O}_0 \otimes \chi)$ is one of the following:
\begin{enumerate}
\item $\mathcal{L}^{-1}_{\chi} \otimes \mathcal{O}_{E_i}$
\item $\mathcal{L}^{-1}_{\chi} \otimes \mathcal{O}_{E_i \cap E_j}$
\item $\mathcal{F} [1] \quad\text{ where }\quad \supp_Y(\mathcal{F}) = 
E_{i_1} \cup \dots \cup E_{i_k}$
\item $\mathcal{O}_{Y}(\Except(Y)) \otimes
\mathcal{O}_{\Except(Y)}[2]$
\end{enumerate} where $E_i$ are irreducible exceptional divisors,
$\mathcal{F}$ is a coherent sheaf and $\L_{\chi} =
\Psi(\mathcal{O}_{\mathbb{C}^3} \otimes \chi^{-1}) = (\pi_{Y *}
\mathcal{M})_{\chi}$ are the tautological bundles on $Y$.  
\end{theorem}

We prove Theorem $\ref{theorem-geometrical-mckay-correspondence}$ by
investigating the relation we uncovered between the geometric McKay
correspondence and a piece of toric geometric combinatorics known as
`Reid's recipe'. It was originally developed by Reid in
\cite{Kinosaki-97} and then employed by Craw in
\cite{Craw-AnexplicitconstructionoftheMcKaycorrespondenceforAHilbC3}
to tackle the problem of finding a basis for $H^*(Y,\mathbb{Z})$
naturally bijective to $\irr(G)$. This problem also has its roots
in \cite{GsV-ConstructionGeometriqueDeLaCorrespondanceDeMcKay} where
it was shown that $\{c_1(\mathcal{L}_\rho)\}_{\rho \in
\irr(G)\setminus\rho_0}$ is the basis of $H^2(Y,\mathbb{Z})$
dual to the basis $\{[E]\}_{E \in \Except(Y)}$ of $H_2(Y,\mathbb{Z})$ 
with $[E_\chi]$ being precisely the vector dual to
$c_1(\mathcal{L}_\chi)$. Taking $\mathcal{L}_{\rho_0} = \mathcal{O}_Y$ to base 
$H^0(Y,\mathbb{Z})$ we obtain a natural basis of $H^*(Y,\mathbb{Z})$
in dimension two. In dimension three
$c_1(\mathcal{L}_\chi)$ still span $H^2(Y,\mathbb{Z})$, but there are
relations. Reid's recipe singles out $\mathcal{L}_\chi$ whose first 
Chern classes are redundant and replaces them by abstract elements 
of $K(Y)$ in such a way that the second Chern
classes of these `virtual' bundles base $H^4(Y,\mathbb{Z})$. The recipe 
is based around a marking which via some simple toric geometric
calculations assigns a character $\chi \in \irr(G)$ to every
exceptional toric curve on $Y$ and then
a character or a pair of characters to every exceptional toric divisor
$E \in \Except(Y)$ (see
\cite{Craw-AnexplicitconstructionoftheMcKaycorrespondenceforAHilbC3},
Section 3 or our short summary of it in the Section
\ref{subsection-reids-recipe} of the present paper). Based on
some more computational evidence we conjectured in 
\cite{CautisLogvinenko-ADerivedApproachToGeometricMcKayCorrespondence}, 
Conj. $1.3$, the relation between the marking of Reid's recipe 
and the transforms $\Psi(\mathcal{O}_0 \otimes \chi)$. A stronger 
version of this conjecture we prove in the present paper. 
The ``only if'' implication of item
(\ref{item-chi-marks-a-divisor}) was proved by Craw and Ishii
in \cite{Craw-Ishii-02}, Prop. 9.3, but the rest are original 
to this paper and for the first time a complete categorification
is obtained which for every possibility for $\chi$ in Reid's recipe 
describes the transform $\Psi(\mathcal{O}_0 \otimes \chi)$ in $D(Y)$: 
\begin{theorem}\label{theorem-reids-recipe-knows-everything}
Let $G \subset \gsl_3(\mathbb{C})$ be a finite abelian subgroup
such that $\mathbb{C}^3/G$ has a single isolated singularity at the
origin. Let $\chi$ be a character of $G$. Then in Reid's recipe
$\chi$ marks
\begin{enumerate}
\item \label{item-chi-marks-a-divisor} a divisor $E \in \Except(Y)$ if 
and only if 
$\Psi(\mathcal{O}_0 \otimes \chi) = \mathcal{L}^{-1}_\chi \otimes 
\mathcal{O}_E$.  
\item \label{item-chi-marks-a-single-curve} a single curve $C$ if 
and only if 
$\Psi(\mathcal{O}_0 \otimes \chi) = \mathcal{L}^{-1}_\chi \otimes 
\mathcal{O}_C$. 
\item \label{item-chi-marks-several-curves} several curves if and only if 
$\Psi(\mathcal{O}_0 \otimes \chi)$ is supported in degree $-1$. 
The support of $H^{-1} \left(\Psi(\mathcal{O}_0 \otimes
\chi)\right)$ is then the union of all $E \in \Except(Y)$ which
contain two or more of the curves marked by $\chi$.
\item \label{item-chi-marks-nothing} nothing (i.e. $\chi$ is
the trivial character $\chi_0$) if and only if
$\Psi(\mathcal{O}_0 \otimes \chi) = \mathcal{O}_Y(\Except(Y)) \otimes 
\mathcal{O}_{\Except(Y)}[2]$. 
\end{enumerate}
\end{theorem}

In this stronger form Theorem \ref{theorem-reids-recipe-knows-everything}
easily implies Theorem \ref{theorem-geometrical-mckay-correspondence}.
To prove Theorem \ref{theorem-reids-recipe-knows-everything} it 
was necessary to relate the derived category 
mechanics of computing the transforms $\Psi(\mathcal{O}_0 \otimes \chi)$
to the toric combinatorics of the markings in Reid's recipe. 
The first step towards this was made with the \em sink-source graphs
\rm of \cite{CautisLogvinenko-ADerivedApproachToGeometricMcKayCorrespondence},
Section $4$. Given an exceptional divisor $E \in \Except(Y)$ the sink-source
graph $SS_{\mathcal{M}, E}$ of $\mathcal{M}$ along $E$ is a graph drawn 
on top of the \em McKay quiver \rm $Q(G)$ whose vertices 
are certain vertices of $Q(G)$ and whose edges are certain 
paths in $Q(G)$. Which vertices
and which paths is determined by the behaviour of the
family $\mathcal{M}$ generically along $E$ (see Section 
$\ref{section-mckay-quiver-and-sink-source-graphs}$).
Something employed throughout
\cite{CautisLogvinenko-ADerivedApproachToGeometricMcKayCorrespondence}
but never stated explicitly was that 
for any $\chi \in \irr(G)$ the divisor $E$ belongs to the support of 
$H^i \Psi(\mathcal{O}_0 \otimes \chi)$ for a certain $i$ if and only if 
$\chi$ is a vertex of $SS_{\mathcal{M}, E}$ of a certain type
(Prps. \ref{sinks-and-sources-iff-in-transforms-support} of the present paper). 

It was shown in
\cite{CautisLogvinenko-ADerivedApproachToGeometricMcKayCorrespondence}
that there are only three possible shapes that the sink-source graph 
$SS_{\mathcal{M}, E}$ of any divisor $E \in \Except(Y)$ can have (see Section 
$\ref{section-mckay-quiver-and-sink-source-graphs}$). On the other hand, in 
Reid's recipe all the divisors in $\Except(Y)$ are divided into three
classes and the marking for a divisor is determined in a different
fashion in each class (see \S 3 of
\cite{Craw-AnexplicitconstructionoftheMcKaycorrespondenceforAHilbC3}
or our account of it in Section $\ref{subsection-reids-recipe}$). 
The crucial step at the heart of the present paper 
is showing that these three classes of the divisors in Reid's recipe
and the three possible shapes of the sink-source graphs of
$\mathcal{M}$ are in exact correspondence (Theorem 
$\ref{theorem-ssgraphs-to-divisor-types}$). Moreover, we can calculate
the monomials which define the toric curves contained in $E$ 
in terms of the lengths of 
the edges of $SS_{\mathcal{M}, E}$\footnote{A length of an edge in
$SS_{\mathcal{M}, E}$ is the number of arrows which compose the
corresponding path in the McKay quiver.} and vice versa (see Prps.
$\ref{prps-(3,3)-source-graph-to-case-1}$-$\ref{prps-3-(1,2)-sources-to-case-3}$). With this established most
of the work needed to prove Theorem
\ref{theorem-reids-recipe-knows-everything} is done as for any 
$E \in \Except(Y)$ and any $\chi \in \irr(G)$ we can translate 
the information that $E \in H^i \Supp \Psi(\mathcal{O}_0 \otimes \chi)$
for some $i$ via the sink-source graph of $E$ into the information on 
the markings of $E$
and of the toric curves $E$ contains. It then remains only to exclude some
cases of $\Supp \Psi (\mathcal{O}_0 \otimes \chi)$ containing extra
curves where it shouldn't, which we do via straightforward calculations 
with long exact sequences in sheaf cohomology. 

\bf Acknowledgments: \rm The author would like to thank Sabin Cautis, 
his co-author on
\cite{CautisLogvinenko-ADerivedApproachToGeometricMcKayCorrespondence}. 
The author would also like to thank Alastair King, Alastair Craw and
Akira Ishii for useful discussions and suggestions. Most of the paper
was written during the author's stay at the University of Liverpool and he
would like to thank everyone at the maths department there and
to thank Prof. Viacheslav Nikulin in particular. The paper was 
completed during the COE-COW conference in the University of Tokyo in 
December 2008 and the author would like to thank the organisers
for inviting him to speak there and thus providing him with a good 
deadline to have this paper finished by. 
\newpage

\section{Preliminaries}
\label{section-preliminaries}

Throughout the paper, we take $G$ to be a finite abelian subgroup
of $\gsl_3(\mathbb{C})$ such that $\mathbb{C}^3/G$ has a single
isolated singularity at the origin.  

\subsection{$G$-$\hilb \mathbb{C}^3$} \label{section-ghilb-and-toric}

As is usual we approach the resolution $Y = G$-$\hilb \mathbb{C}^3$
via the methods of toric geometry. For detailed explanation of this
see
\cite{Logvinenko-Families-of-G-constellations-over-resolutions-of-quotient-singularities},
Section 3.1 or \cite{Craw02}. In brief, let $\mathbb{Z}^3$ be the
lattice of Laurent monomials, where we identify point $m = (m_1, m_2,
m_3)$ with the monomial $x_1^{m_1} x_2^{m_2} x_3^{m_3}$. Let $M
\subset \mathbb{Z}^3$ be the sublattice of $G$-invariant monomials. 
Dually, we have the overlattice $(\mathbb{Z}^3)^\vee \subset L$, known 
as the lattice of weights. As $G$ is finite we have 
$L \subset \mathbb{Q}^3$ and we think of any point $l \in L$ as of 
a triplet $(l_1, l_2, l_3) \in \mathbb{Q}^3$. 

Let $\sigma_+$ be the positive octant cone in 
$L \otimes \mathbb{R} = \mathbb{R}^3$ defined by $\{(e_i) \in
\mathbb{R}^3  \;|\; e_i \geq 0\}$. The original affine space $\mathbb{C}^3$ 
is defined as a toric variety by a single cone $\sigma_+$ and the
lattice $\mathbb{Z}^3$. The singular quotient $\mathbb{C}^3/G$ is 
defined by $\sigma_+$ and the lattice $L$. The crepant resolutions
of $\mathbb{C}^3/G$ are defined in the lattice $L$ by the fans 
which subdivide the cone $\sigma_+$ into regular subcones.
Let $\Delta$ denote the section of $\sigma_+$ by 
the hyperplane $\sum e_i = 1$ in $\mathbb{R}^3$. It is a regular
triangle which we call \it the junior simplex\rm. 
We identify the subdivisions of $\sigma_+$ into regular subcones with
the corresponding triangulations of the junior simplex $\Delta$.

It is described in
\cite{Craw-AnexplicitconstructionoftheMcKaycorrespondenceforAHilbC3},
Section 2, how to construct the triangulation $\Sigma$ of $\Delta$ whose
corresponding fan gives the crepant resolution $Y = G$-$\hilb \mathbb{C}^3$. 
Denote this fan by $\mathfrak{F}$. To each $k$-dimensional cone $\sigma$ 
in $\mathfrak{F}$ corresponds a $(3-k)$-dimensional torus orbit $S_\sigma$ 
in $Y$. Denote by $E_\sigma$ the subscheme of $Y$ given by the closure 
of $S_\sigma$, then $E_\sigma$ is the union of $S_{\sigma'}$ for all
cones $\sigma'$ which contain $\sigma$ as a face.  
Denote by $\mathfrak{E}$ the set $L \cap \Delta$, these are the 
vertices of the triangles in $\Sigma$ and, correspondingly, the generators of 
the one-dimensional cones in $\mathfrak{F}$ . Then 
$\{E_{e}\}_{e \in \mathfrak{E}}$, where we write $E_{e}$ for 
$E_{\left< e \right>}$, are precisely the 
exceptional divisors of $Y$ together with strict transforms of 
the hyperplanes $x^{|G|} = 0$, $y^{|G|} = 0$ and $z^{|G|} = 0$
in $\mathbb{C}^3/G$. 

Two-dimensional cones in $\mathfrak{F}$ are the sides of the triangles 
in $\Sigma$. For any
$e,f \in \mathfrak{E}$ the cone $\left< e,f \right>$ lies
in $\mathfrak{F}$ if and only if the exceptional divisors $E_e$ and 
$E_f$ intersect. The orbit closure $E_{e,f}$ is precisely 
the intersection $E_e \cap E_f$ and it is always a $\mathbb{P}^1$. 

Three-dimensional cones in $\mathfrak{F}$ are the triangles in 
$\Sigma$. For any such cone $\sigma$
we denote by $A_\sigma$ the toric affine chart which consists of
the torus orbits corresponding to all cones in $\mathfrak{F}$ which 
are faces of $\sigma$. We have a natural isomorphism 
$A_\sigma \simeq \mathbb{C}^3$ which maps the torus fixed point 
$E_\sigma$ to the origin $0 \in \mathbb{C}^3$.  

\subsection{Reid's recipe}
\label{subsection-reids-recipe}

Let $\mathcal{M}$ be the universal family of $G$-clusters on $Y \times
\mathbb{C}^3$. It is a $G$-equivariant coherent sheaf on $Y \times
\mathbb{C}^3$. The category of $G$-equivariant quasi-coherent 
sheaves on $Y \times \mathbb{C}^3$ is equivalent to the category of 
quasi-coherent $(\twalg)\otimes_\mathbb{C} \mathcal{O}_Y$-modules on $Y$
via the pushdown functor. So we shall often abuse the notation by
identifying $\mathcal{M}$ with its pushdown to $Y$ considered as an
$(\twalg)\otimes_\mathbb{C} \mathcal{O}_Y$-module.  

The family $\mathcal{M}$ is defined up to an equivalence 
of families, that is - up to a twist by a ($G$-invariant) line bundle on $Y$. 
We normalize by assuming the line bundle $\left[\mathcal{M}\right]^G$ to be 
trivial and identifying it with $\mathcal{O}_Y$. This uniquely 
determines $\mathcal{M}$ up to isomorphism.  Write the decomposition
of $\mathcal{M}$ into $G$-eigensheaves as $\mathcal{M} =
\bigoplus_{\chi \in G^\vee} \mathcal{L}_\chi$ where $G$ acts on
$\mathcal{L}_\chi$ by $\chi$. The sheaves $\mathcal{L}_\chi$ are line
bundles on $Y$ known in the literature as \em tautological \rm or \em
Gonzales-Sprinberg and Verdier \em sheaves. 

Reid's recipe (\cite{Kinosaki-97},
\cite{Craw-AnexplicitconstructionoftheMcKaycorrespondenceforAHilbC3})
is an algorithm to construct the cohomological version of the McKay
correspondence. It takes the Chern classes of the tautological sheaves 
and modifies some of them to turn them into a basis of the cohomology 
ring $H^\bullet(Y,\mathbb{Z})$. It consists of two parts: in the first 
each edge and each vertex in the triangulation $\Sigma$ are marked
by a character of $G$ in accordance with the geometry of the toric 
fan $\mathfrak{F}$ of $Y$. In the second the data of this marking
is used to dictate the way in which the Chern classes of the
tautological sheaves are modified to produce a basis of
$H^\bullet(Y,\mathbb{Z})$. This cohomological construction 
is not relevant to this paper - it is replaced by the version 
of the McKay correspondence proposed in 
\cite{CautisLogvinenko-ADerivedApproachToGeometricMcKayCorrespondence}
where we look at the images of $\mathcal{O}_0 \otimes \chi$ under 
the derived equivalence of \cite{BKR01}. 
However, the data of the marking constructed in the first half of 
the Reid's recipe turns out to dictate the way our correspondence
goes as well -- as demonstrated by our main result, Theorem 
\ref{theorem-reids-recipe-knows-everything}. In a sense, 
our correspondence takes the same source data of the Reid's recipe
marking but realises it in the derived category instead of 
the cohomology ring. 

Below we give a brief summary of the construction of this marking. 
Denote by $\regring$ the coordinate ring $\mathbb{C}[x,y,z]$ of
$\mathbb{C}^3$. First we mark 
each edge $(e,f)$ in the triangulation $\Sigma$ with a character of 
$G$ according to the following rule. The one-dimensional ray in $M$
perpendicular to the hyperplane $\left< e,f\right>$ in $L$ has two
primitive generators: $\frac{m_1}{m_2}$ and $\frac{m_2}{m_1}$, where
$m_1, m_2$ are co-prime regular monomials in $\regring$. 
As $M$ is the lattice of $G$-invariant Laurent monomials, $m_1$ and $m_2$ have
to be of the same character $\chi$ for some $\chi \in G^\vee$. We say
that $(e,f)$ is \it carved out \rm by the ratio $m_1 : m_2$
(or $m_2 : m_1$) and \it mark it \rm by $\chi$. 

Then we mark the vertices of $\Sigma$ according to a recipe which 
is based on the following classification:
 
\begin{prps}[\cite{Craw-AnexplicitconstructionoftheMcKaycorrespondenceforAHilbC3},
see \S2-\S3]
\label{prps-the-classification-of-the-vertices-of-the-triangulation}
For any $e \in \mathfrak{E}$ the corresponding vertex in the triangulation 
$\Sigma$ is one of the following:
\begin{enumerate}
\item \label{case-meeting-of-champions} A meeting point of three lines emanating from the three vertices of
$\Delta$, as depicted on Figure \ref{figure-01}. 

\item \label{case-one-line-from-vertex}
An interior point of exactly one line emanating from a vertex of
$\Delta$. Other than the two edges coming from this line, it also 
has $2$, $3$ or $4$ other edges incident to it, as depicted
on Figure \ref{figure-02} (up to permutation of $x$, $y$ and $z$).

\item \label{case-three-straight-lines} An intersection point of three
straight lines none of which emanate from a vertex of $\Delta$, 
as depicted on Figure \ref{figure-05}. 
\end{enumerate}
\end{prps}
\begin{figure}[!h] \begin{center}
\includegraphics[scale=0.24]{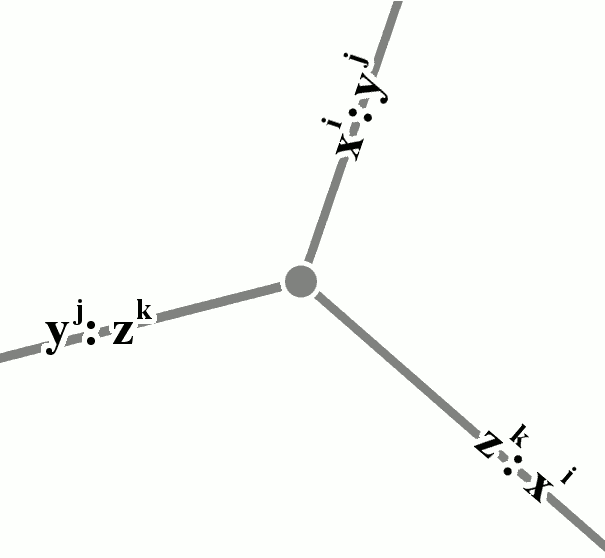} \end{center}
\caption{\label{figure-01} Case \ref{case-meeting-of-champions}}
\end{figure}
\begin{figure}[!h] \centering 
\subfigure[Valency 4] { \label{figure-02a}
\includegraphics[scale=0.25]{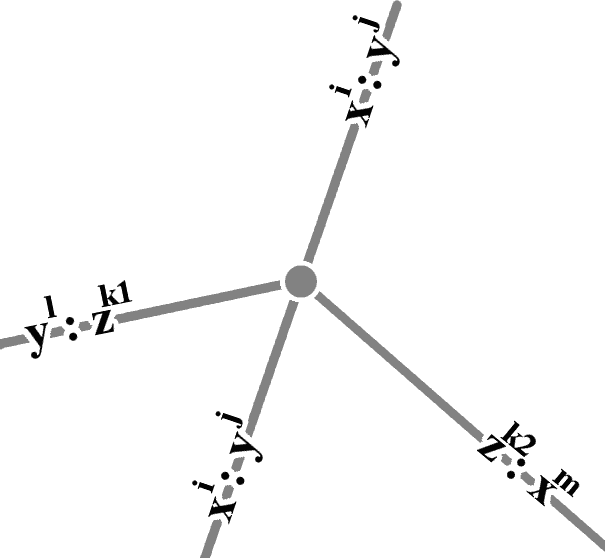}} 
\subfigure[Valency 5] { \label{figure-02b}
\includegraphics[scale=0.25]{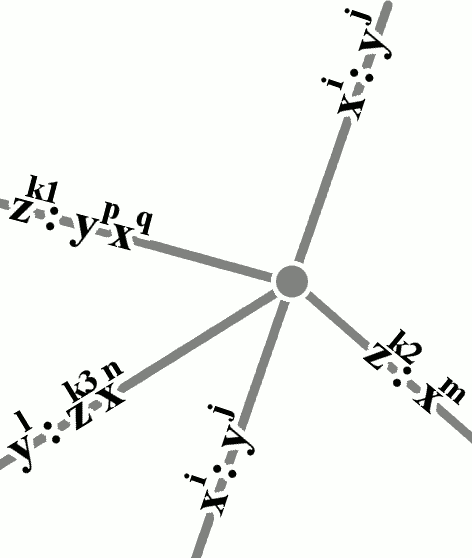}} 
\subfigure[Valency 6] { \label{figure-02c}
\includegraphics[scale=0.25]{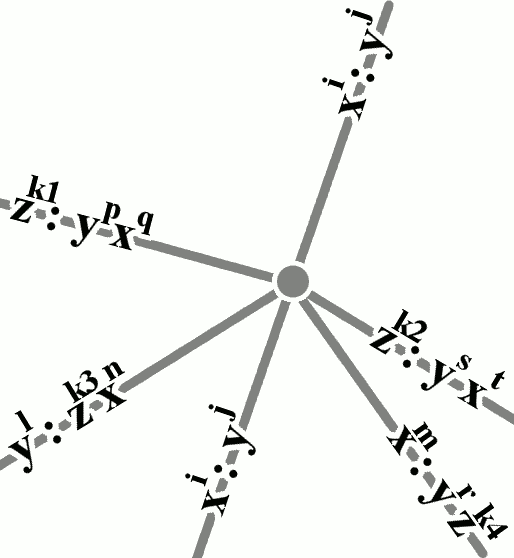}}
\caption{Case \ref{case-one-line-from-vertex}} \label{figure-02}
\end{figure}
\begin{figure}[!htb] \begin{center}
\includegraphics[scale=0.26]{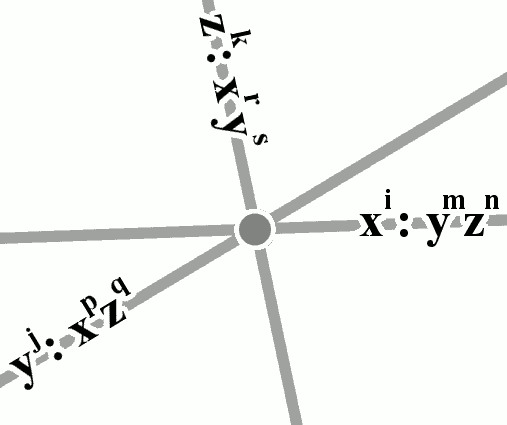} \end{center}
\caption{\label{figure-05} Case \ref{case-three-straight-lines}}
\end{figure}
When vertex $e \in \mathfrak{E}$ belongs to the Case
\ref{case-meeting-of-champions} it is clear 
from Figure \ref{figure-01} that all the three edges incident to 
$e$ are marked with the same character $\chi \in G^\vee$, 
which is the common character of $x^i, y^j$ and $z^k$. Reid's recipe
prescribes for such $e$ to be marked with the character $\chi \cdot \chi$.  

When $e$ belongs to the Case \ref{case-one-line-from-vertex} it is proved in
\cite{Craw-AnexplicitconstructionoftheMcKaycorrespondenceforAHilbC3}, 
Lemmas $3.2$-$3.3$ that $k_1 = k_2$. Reid's recipe prescribes
for such $e$ to be marked with the character $\chi \cdot \chi'$, 
where $\chi$ is the common character of $y^j$ and $z^k$, 
i.e. the character which marks the unique line which emanates from 
one of the vertices of $\Delta$ and contains $e$ as interior point, and $\chi'$ 
is the character of $x^{k_1} = x^{k_2}$, i.e. the character which 
marks precisely two of the remaining edges incident to $e$. 

Finally, when $e$ belongs to the Case \ref{case-three-straight-lines}
it is proved in \cite{Craw-AnexplicitconstructionoftheMcKaycorrespondenceforAHilbC3},
Lemma $3.4$ that the monomials $x^i z^q$, $y^j x^r$ and $z^k y^m$
are all of the same character. Denote it by $\chi \in G^\vee$. 
It is also proved that the monomials $x^i y^s$, $z^k x^p$ and $y^j z^n$ 
are also all of the same character. Denote it by $\chi' \in G^\vee$. 
Reid's recipe prescribes for such $e$ to be marked by two characters - 
$\chi$ and $\chi'$.  

\subsection{$G$-clusters and $G$-graphs} 
\label{section-G-clusters-and-g-graphs}

The resolution map $Y \rightarrow \mathbb{C}^3/G$ induces
an inclusion of $K(Y)$ into $K(\mathbb{C}^3)$ and thus allows 
to view $K(\mathbb{C}^3)$ as a constant sheaf of $(\twalg)
\otimes_\mathbb{C} \mathcal{O}_Y$-modules on $Y$. As shown in 
\cite{Logvinenko-Natural-G-Constellation-Families}, \S 3.1, there 
is a unique $(\twalg) \otimes_\mathbb{C} \mathcal{O}_Y$-module  
embedding of $\mathcal{M}$, normalized as above, into $K(\mathbb{C}^3)$ 
which maps $\mathcal{O}_Y \subset \mathcal{M}$ identically to 
$\mathcal{O}_Y \subset K(\mathbb{C}^3)$. Then for each character $\chi \in
G^\vee$ the image of the embedding $\mathcal{L}_\chi \hookrightarrow 
K(\mathbb{C}^3)$ is $\mathcal{L}(-D_\chi)$ for some uniquely defined
fractional $G$-Weil divisor $D_{\chi} =  \sum_{f \in \mathfrak{E}}
q_{\chi,f} E_f$ (\cite{Logvinenko-Natural-G-Constellation-Families}, \S2).
Thus $\mathcal{M}$ can be written canonically as 
$\bigoplus_{\chi \in G^\vee} \mathcal{L}(-D_\chi)$.
 
Let $\sigma$ be a three-dimensional cone in $\mathfrak{F}$ generated
by some $i,j,k \in \mathfrak{E}$. Then $\mathcal{L}(-D_\chi)$ is 
generated inside $K(\mathbb{C}^3)$ on the affine chart $A_\sigma$ by 
the unique Laurent monomial $r$ for which 
\begin{align}
i(r) = q_{\chi,i}, \quad j(r) = q_{\chi,j}, 
\quad k(r) = q_{\chi,k} . 
\end{align}
This is natural, considering that for any $i \in \mathfrak{E}$ and 
any Laurent monomial $m \in \mathbb{Z}^3$ we have 
$i(m) = \text{val}_{E_i}(m)$ with the RHS being a $\mathbb{Q}$-valuation 
defined as $\frac{1}{|G|} \text{val}_{E_i} m^{|G|}$, where $G$-invariant
Laurent monomial $m^{|G|}$ is treated as a rational function on $Y$
and is valuated at prime Weil divisor $E_i$ (\cite{Logvinenko-Families-of-G-constellations-over-resolutions-of-quotient-singularities}, Prps. 3.2). 

Let $\sigma$ be a three-dimensional cone in $\mathfrak{F}$. The set
$\Gamma_{\sigma} = \{r_{\chi}\}_{\chi \in G^\vee}$, where 
$r_\chi$ is the unique monomial generator of $\mathcal{L}(-D_\chi)$ 
over the affine chart $A_\sigma$, is called 
the \it $G$-graph \rm of $A_\sigma$. The monomials in 
$\Gamma_\sigma$ are precisely the monomials which do not lie in 
the ideal $I_\sigma \subset \regring$ defining the $G$-cluster parametrised 
by the torus fixed point $E_\sigma$ of the chart $A_\sigma$. 

Let $(e, f)$ be any edge in the triangulation $\Sigma$. Let 
$(e, f, g)$ and $(e, f, g')$ be the two triangles containing it
and let $\sigma$ and $\sigma'$  be the corresponding
three-dimensional cones in $\mathfrak{F}$. Let 
$\Gamma_{\sigma} = \{r_\chi\}_{\chi \in G^\vee}$ 
and $\Gamma_{\sigma'} = \{ r'_\chi \}_{\chi \in G^\vee}$ be 
the $G$-graphs of affine toric charts $A_{\sigma}$ and $A_{\sigma'}$.
Suppose that the hyperplane $\left<e,f\right>$ in $L$ 
is carved out by the ratio $m : m'$ for some co-prime regular monomials 
$m, m'$ in $\regring$. Suppose, without loss of generality, that
$g(\frac{m'}{m}) > 0$. The following simple observation is fundamental 
in explaining the link between Reid's recipe and the geometrical
McKay correspondence: 

\begin{lemma} \label{lemma-markings-divide-nonequal-ggraph-pieces}
If $r_\chi \neq r'_\chi$ for some $\chi \in G^\vee$
then $m | r_\chi$ and $m' | r'_\chi$.
\end{lemma}
\begin{proof}
Since $r_\chi$ and $r'_\chi$ generate $\mathcal{O}_Y$-module
$\mathcal{L}(-D_\chi)$ on toric affine charts $A_\sigma$ and
$A_{\sigma'}$, respectively, $\frac{r'_\chi}{r_\chi}$ has to be 
invertible on $A_\sigma \cap A_\sigma'$. Therefore the valuation of 
$\frac{r'_\chi}{r_\chi}$ is zero on $E_e$ and $E_f$. In other
words $\frac{r'_\chi}{r_\chi} \in \left<e,f\right>^\perp$. Since 
$\left<e,f\right>^\perp$ is a one-dimensional ray in $M$ generated 
by $\frac{m'}{m}$ we must have $\frac{r'_\chi}{r_\chi} =
(\frac{m'}{m})^k$ for some $k \in \mathbb{Z}$. 

On the other hand, recall that we have $\mathcal{M}^G =
\mathcal{O}_Y$ and therefore $1$ is a global section of $\mathcal{M}$. 
So is any $f \in \regring$ as $\regring$ acts on $\mathcal{M}
\subset K(\mathbb{C}^3)$ by restriction of the natural action 
of $\regring$ on $K(\mathbb{C}^3)$ by multiplication
and so $f = f \cdot 1$. In particular, both $r_\chi$ and
$r'_\chi$ are sections of $\mathcal{L}(-D_\chi)$ on $A_\sigma$. Since 
$r_\chi$ generates $\mathcal{L}(-D_\chi)$ on 
$U_\sigma$ as an $\mathcal{O}_Y$-module we must have 
$\frac{r'_\chi}{r_\chi} \in \mathcal{O}_Y(U_\sigma)$
and hence $g(\frac{r'_\chi}{r_\chi}) = \text{val}_{E_g}(\frac{r'_\chi}{r_\chi}) 
\geq 0$. As we assumed that $g(\frac{m'}{m}) > 0$, we conclude 
that $k \geq 0$. And as by assumption $r_\chi \neq r'_\chi$
we further have $k >0$. The claim now follows. 
\end{proof}

\begin{cor} \label{cor-markings-belong-to-ggraphs}
Let $\chi \in G^\vee$ be the character of $G$ which marks 
$\left< e, f \right>$, i.e. $\chi$ is the common character of $m$ and $m'$. 
Then $r_\chi = m$, whilst $r'_\chi = m'$. 
\end{cor}

\begin{proof}
There has to exist $\chi' \in G^\vee$ for which $r_{\chi'} \neq
r'_{\chi'}$ as otherwise the $G$-clusters parametrised
by the torus fixed points $E_\sigma$ and $E_\sigma'$ would be
isomorphic. By Lemma $\ref{lemma-markings-divide-nonequal-ggraph-pieces}$ 
we must then have $m | r_{\chi'}$ and $m' | r_{\chi'}$. But any regular
monomial which divides an element of a $G$-graph must itself 
belong to that $G$-graph, as the compliment of a $G$-graph in 
a set of regular monomials is the monomial part of an ideal in 
$\regring$. The claim follows.
\end{proof}

\begin{cor} \label{cor-primitivity-of-marking-ratios}
There doesn't exist a $G$-invariant Laurent monomial 
$\frac{m_1}{m'_1} \neq 1$ such that $m_1$ and $m'_1$ strictly 
divide $m$ and $m'$ respectively.  
\end{cor}
\begin{proof}
Suppose such $\frac{m_1}{m'_1}$ were to exist. By Corollary
$\ref{cor-markings-belong-to-ggraphs}$ we have $m \in \Gamma_\sigma$ 
and as $m_1 \mid m$ we must also have $m_1 \in \Gamma_\sigma$. 
Similarly, we must have $m'_1 \in \Gamma_{\sigma'}$. Since
$\frac{m_1}{m'_1}$ is $G$-invariant the monomials $m_1$ and $m'_1$
have to be of the same character and as $\frac{m_1}{m'_1} \neq 1$
we must have $m_1 \neq m'_1$. But then by Lemma
$\ref{lemma-markings-divide-nonequal-ggraph-pieces}$ we 
must have $m \mid m_1$ and $m' \mid m'_1$, contradicting the 
assumption that $m_1$ and $m'_1$ strictly 
divide $m$ and $m'$.
\end{proof}

\subsection{The McKay quiver of $G$ and the sink-source graphs}
\label{section-mckay-quiver-and-sink-source-graphs}

A detailed account of this is given in
\cite{CautisLogvinenko-ADerivedApproachToGeometricMcKayCorrespondence},
Section 4. Below we briefly summarize the essentials. 

The action of $G$ on $\regring$ is obtained from the action of $G$ on 
$\mathbb{C}^3$ by setting $g \cdot m (\bf v\rm) = m (g^{-1} \cdot \bf
v\rm)$ for all $m \in \regring$ and $\bf v\rm \in \mathbb{C}^3$.  
For any regular monomial $m \in \regring$ denote by $\kappa(m)$ the
character with which $G$ acts on $m$. Quite generally, to any finite 
subgroup $G \subset \gl_n(C)$ we can associate a quiver $Q(G)$ called 
\it the McKay quiver of $G$\rm. In our case of a finite abelian subgroup of
$\gsl_3(\mathbb{C})$ the quiver $Q(G)$ has as its vertices the
characters $\chi \in G^\vee$ of $G$ and from every vertex $\chi$
there are three arrows going to $\kappa(x) \chi$, $\kappa(y) \chi$ and
$\kappa(z) \chi$. We denote these arrows by $(\chi,x)$, $(\chi,y)$
and $(\chi,z)$ and say that they are $x$-, $y$- and $z$-oriented,
respectively.

There exists a standard planar embedding of $Q(G)$ into a real two
dimensional torus first constructed by Craw and Ishii in
\cite{Craw-Ishii-02}. We use the version of it detailed in 
\cite{CautisLogvinenko-ADerivedApproachToGeometricMcKayCorrespondence},
Section 4.1. The torus, which we denote by $T_G$, is tesselated
by the embedded $Q(G)$ into $2|G|$ regular triangles. Locally, this
tesselation looks as depicted on Figure \ref{figure-07}. 
\begin{figure}[h] \begin{center}
\includegraphics[scale=0.20]{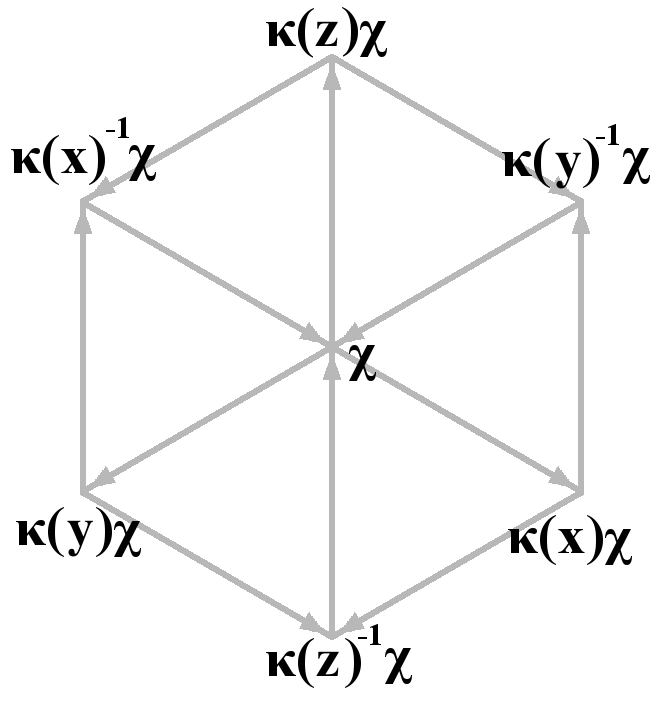} \end{center}
\caption{\label{figure-07} The tesselation of $T_G$ by $Q(G)$ (locally)}
\end{figure}
When depicting $T_G$ in diagrams we draw its fundamental domain 
in $\mathbb{R}^2$. For an example see Figure \ref{figure-08} where
we give a depiction of the McKay quiver of $G = \frac{1}{13}(1,5,7)$ embedded
into $T_G$. $\frac{1}{13}(1,5,7)$ is a common shorthand for the image
in $\gsl_3(\mathbb{C})$ of the group of 13th roots of unity  
under the embedding $\xi
\mapsto 
\left( \begin{smallmatrix}
\xi^1 & & \\
& \xi^5 & \\
& & \xi^7
\end{smallmatrix} \right) $.
\begin{figure}[h] \begin{center}
\includegraphics[scale=0.30]{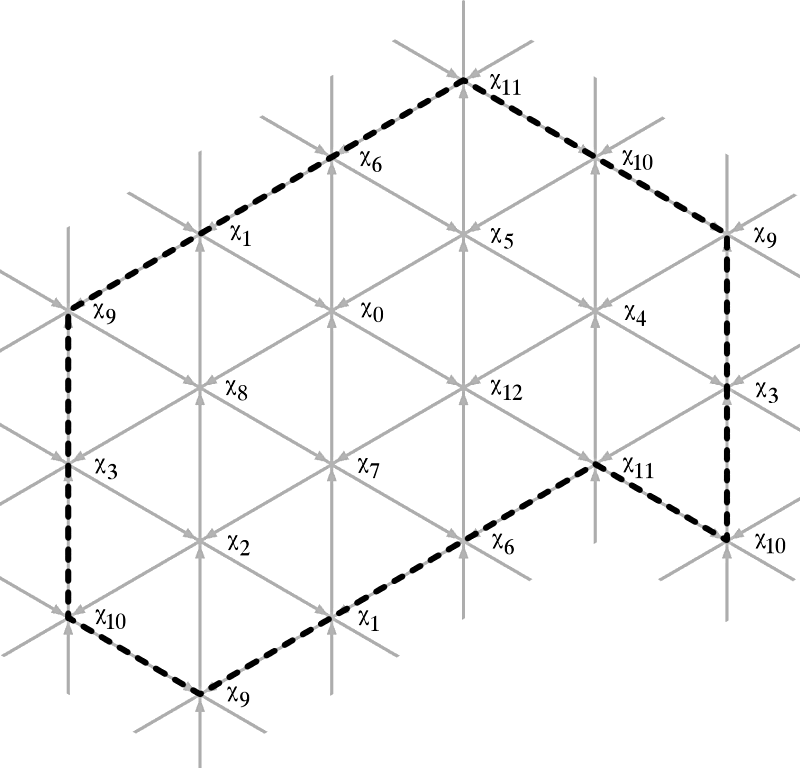} \end{center}
\caption{\label{figure-08} The McKay quiver of $G = \frac{1}{13}(1,5,7)$}
\end{figure}

The importance of the McKay quiver for us is due to the fact 
that $G$-clusters are a special case of a more general concept of
\it $G$-constellations\rm, which are coherent $G$-sheaves on 
$\mathbb{C}^3$ whose global sections are the regular representation 
$\regrep$. The category of $G$-constellations is equivalent
to the category of $\twalg$-modules whose underlying
$G$-representation is $\regrep$ and the latter category is 
equivalent to the category of representations of the McKay quiver into 
the graded vector space $\regrep =  \oplus_{\chi} \mathbb{C}_\chi$ 
where $\mathbb{C}_\chi$ is a copy of $\mathbb{C}$ on which $G$ acts by
$\chi$. This equivalence enables us to define
for the universal family $\mathcal{M}$ of $G$-clusters
(or, more generally, for any 
$\gnat$-family\footnote{A \em $\gnat$-family is a flat family of
$G$-constellations satisfying certain geometrical naturality criteria. See
\cite{Logvinenko-DerivedMcKayCorrespondenceViaPureSheafTransforms},
\S3.2.}) its \it associated representation $Q(G)_\mathcal{M}$ \rm of 
the McKay quiver over $Y$\rm (see
\cite{CautisLogvinenko-ADerivedApproachToGeometricMcKayCorrespondence}, 
Section 4.2). To any arrow $(\chi, x)$ of $Q(G)$ 
in this representation corresponds a map $\alpha_{\chi,x}$ from
$\mathcal{L}(-D_{\chi^{-1}})$, the $\chi$-eigensheaf of $\mathcal{M}$, 
to $\mathcal{L}(-D_{\kappa(x_k)^{-1} \chi^{-1}})$, the $\kappa(x)
\chi$-eigensheaf of $\mathcal{M}$. This map is given by $s \mapsto x \cdot s$.
Denote by $B_{\chi,x}$ the locus in $Y$ where map 
$\alpha_{\chi,x}$ vanishes. It follows from 
\cite{CautisLogvinenko-ADerivedApproachToGeometricMcKayCorrespondence}, 
Propositions $4.4$ and $4.5$ that $B_{\chi,x}$ is an effective
divisor of form $\sum_{e \in \mathfrak{E}} b_e E_e$ where 
$b_e \in \{0,1\}$. We say that
arrow $(\chi,x)$ vanishes along $E_e$ if so does the corresponding 
map $\alpha_{\chi,x}$ in the associated representation 
$Q(G)_{\mathcal{M}}$, i.e. $E_e \subset B_{\chi,x}$. Similarly for the
arrows $(\chi, y)$ and $(\chi,z)$.

Let $E \in \Except{Y}$. For every character $\chi \in G^\vee$ 
we classify the corresponding vertex of $Q(G)$ according 
to which arrows in the subquiver $\hex(\chi)$, formed by the six
triangles containing $\chi$ as per Figure $\ref{figure-07}$, 
vanish along the divisor $E$ and which do not
(\cite{CautisLogvinenko-ADerivedApproachToGeometricMcKayCorrespondence},
Prps. 4.7). On Figures \ref{figure-09} - \ref{figure-13} we list all 
possible cases, drawing in black the arrows which vanish and in grey 
the arrows which don't. These cases divide into four basic classes :
the \it charges\rm, the \it sources\rm, the \it sinks \rm and the \it tiles\rm.
The reason for this choice of names is that charge vertices always occur in
$Q(G)$ in straight lines propagating from a source vertex to a sink
vertex. An $x$-oriented charge propagates along $x$-oriented arrows of
$Q(G)$ and similarly for $y$ and $z$. A type $(1,0)$-charge propagates
in the direction of the arrows, while a type $(0,1)$-charge propagates
against the direction of the arrows. A type $(a,b)$-source (resp.
sink) emits (resp.  receives) $a$ charges of type $(1,0)$ and
$b$-charges of type $(0,1)$. 
\begin{figure}[h]
\begin{center} \includegraphics[scale=0.090]{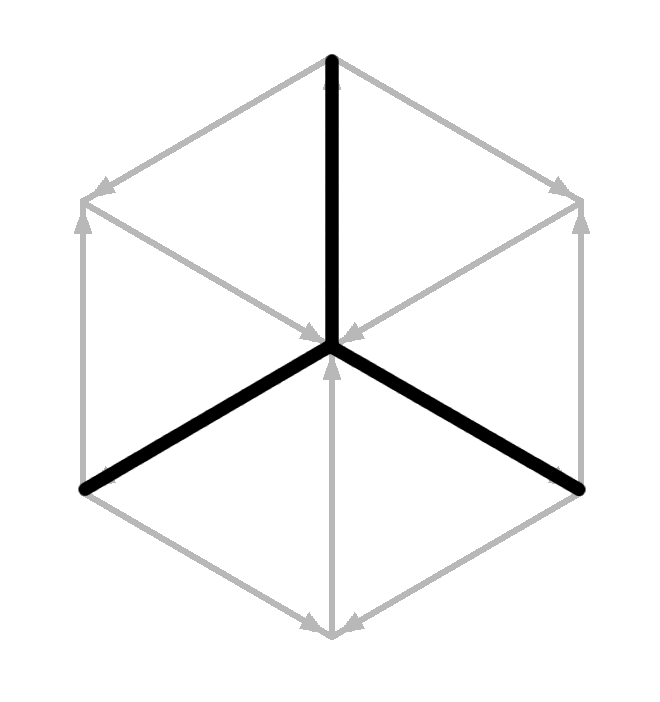}
\includegraphics[scale=0.090]{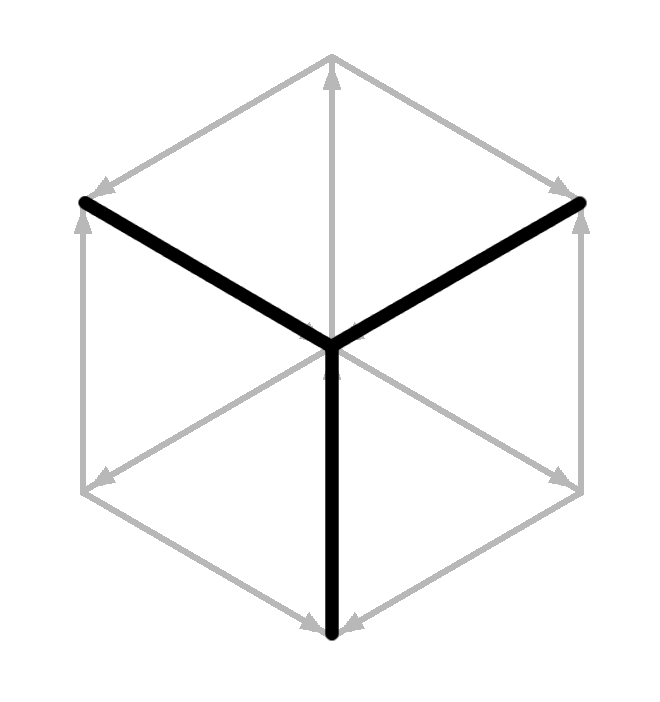}
\caption{\label{figure-09} The $(3,0)$-sink and the $(0,3)$-sink.}
\end{center} \begin{center}
\includegraphics[scale=0.090]{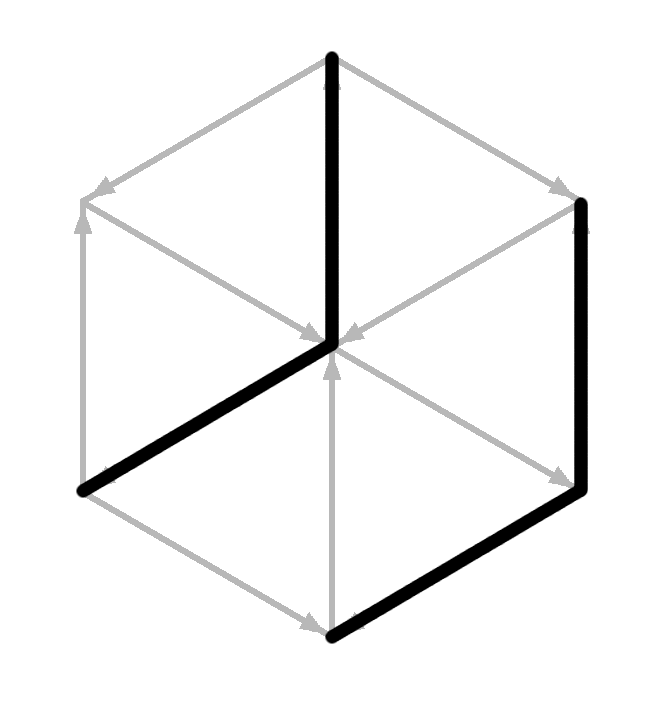}
\includegraphics[scale=0.090]{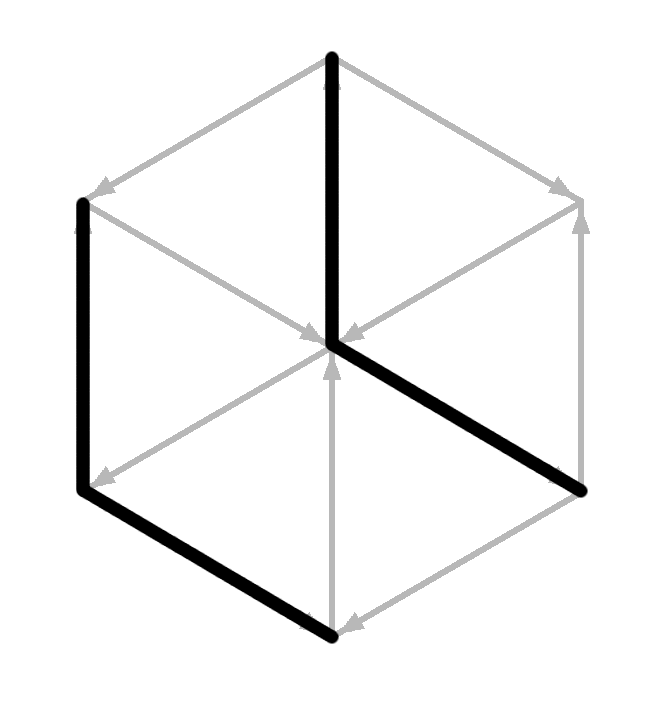}
\includegraphics[scale=0.090]{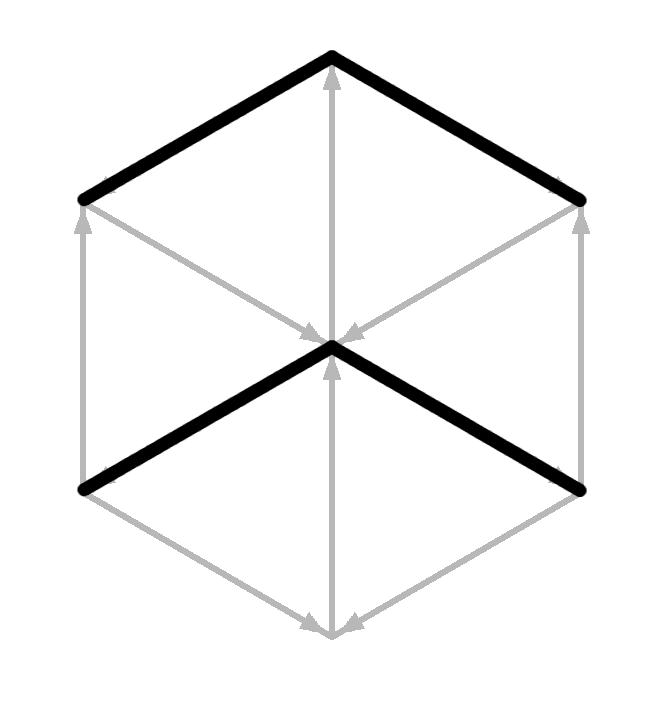}
\includegraphics[scale=0.090]{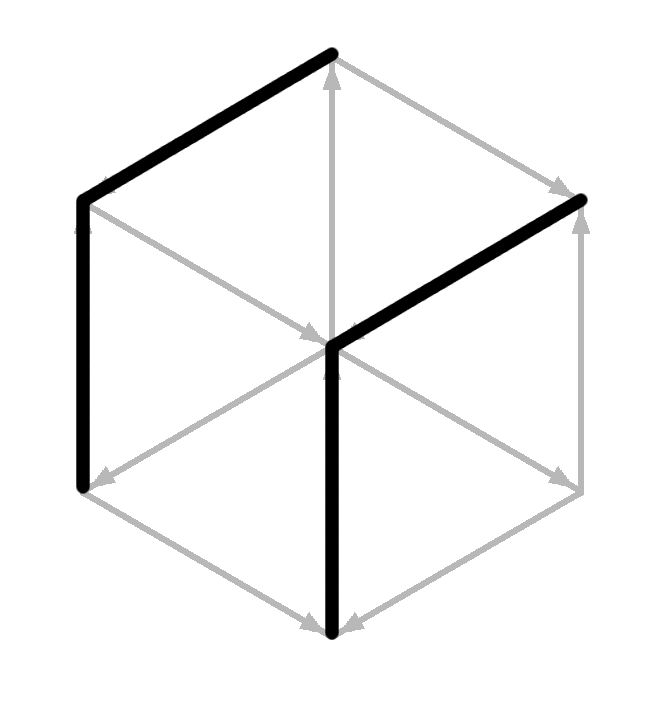}
\includegraphics[scale=0.090]{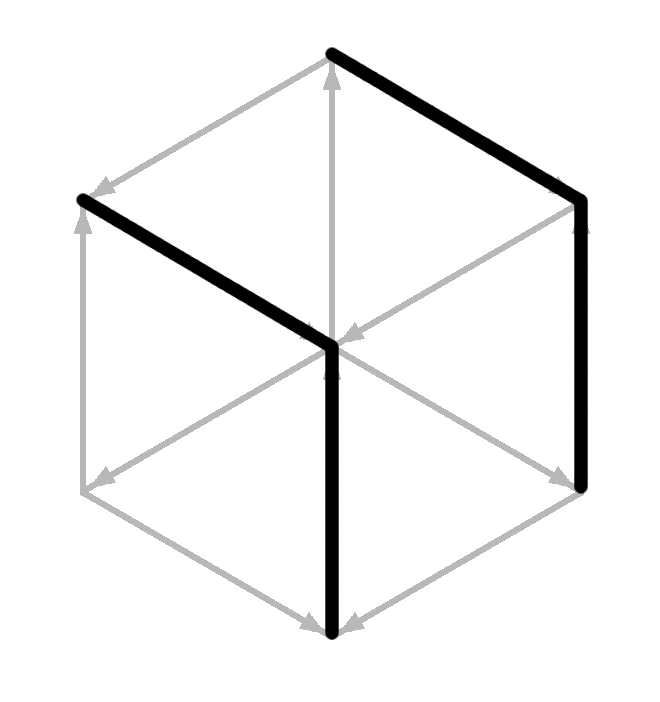}
\includegraphics[scale=0.090]{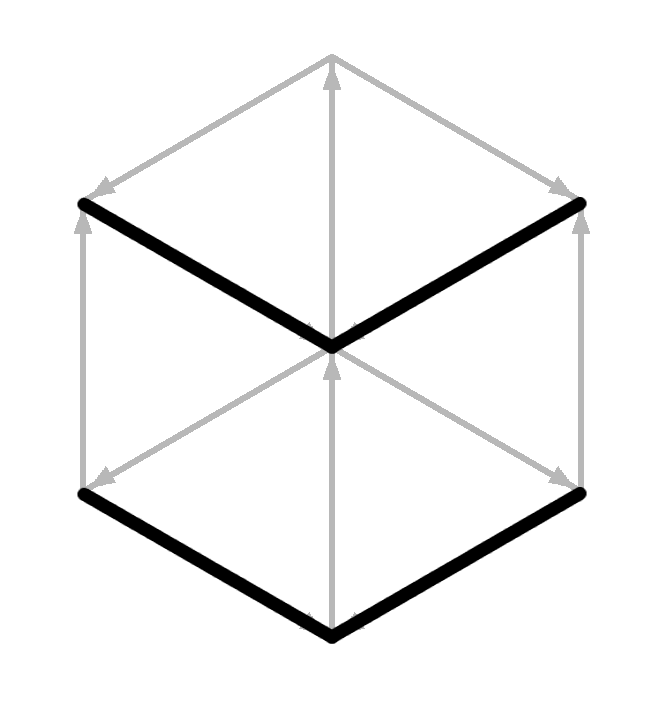}
\caption{\label{figure-10} The $x$-, $y$- and $z$-
$(1,0)$-charges, the $x$-, $y$- and $z$- $(0,1)$-charges}
\end{center} \begin{center}
\includegraphics[scale=0.090]{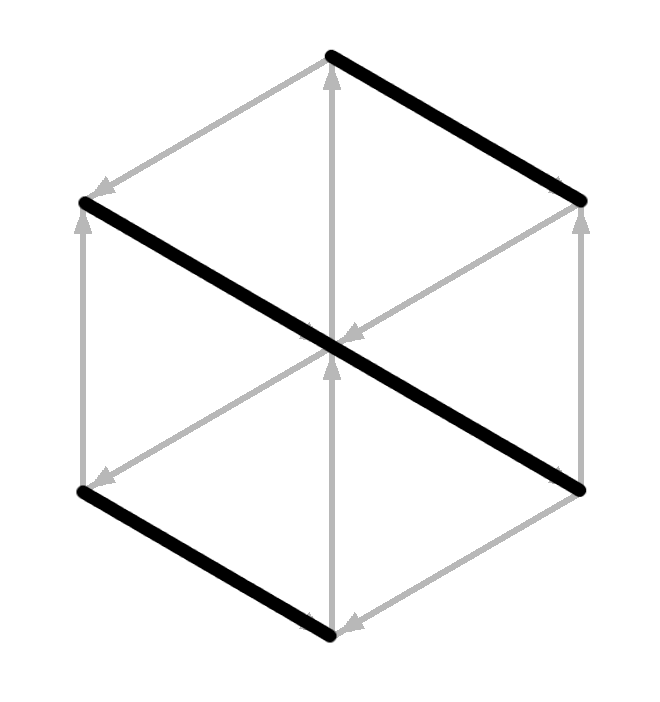}
\includegraphics[scale=0.090]{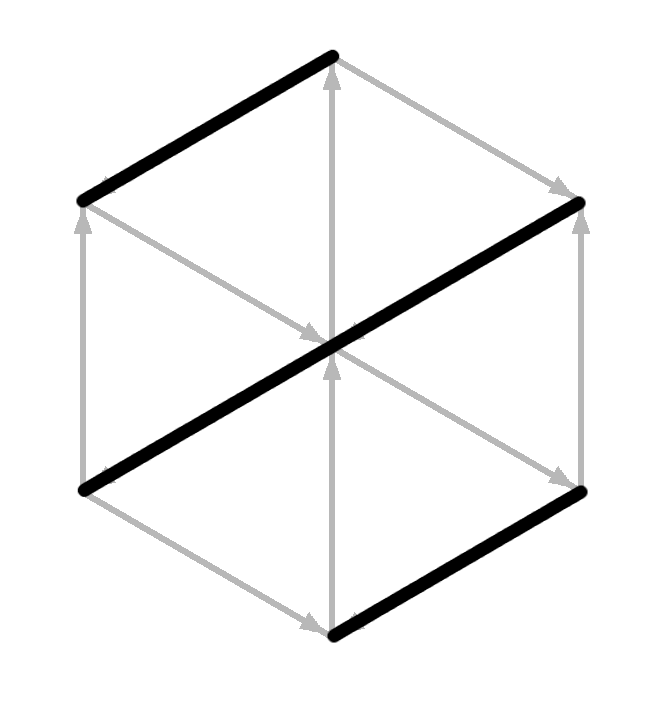}
\includegraphics[scale=0.090]{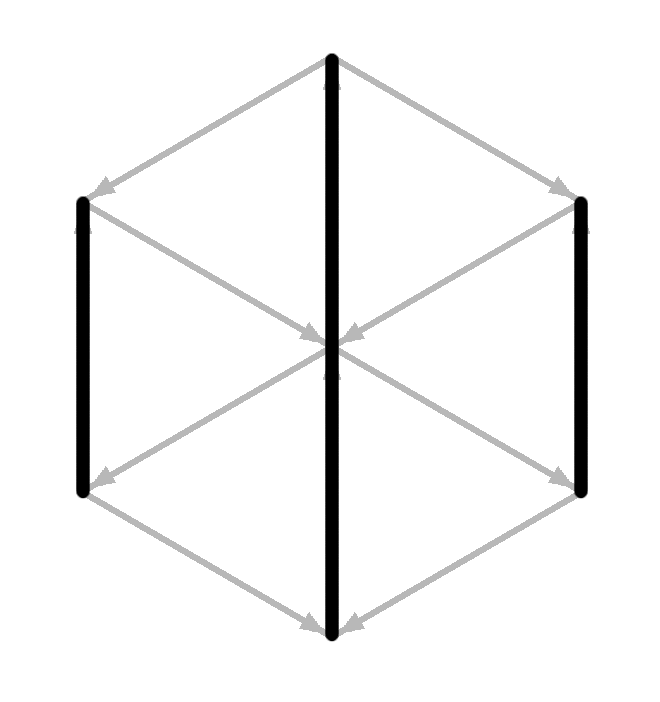}
\caption{\label{figure-11} The $x$-tile, the $y$-tile and the
$z$-tile.} \end{center} \begin{center}
\includegraphics[scale=0.090]{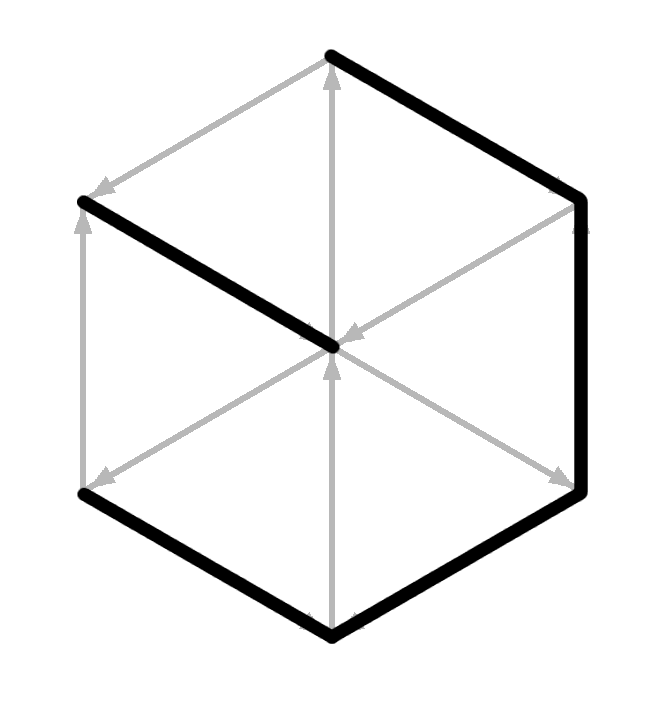}
\includegraphics[scale=0.090]{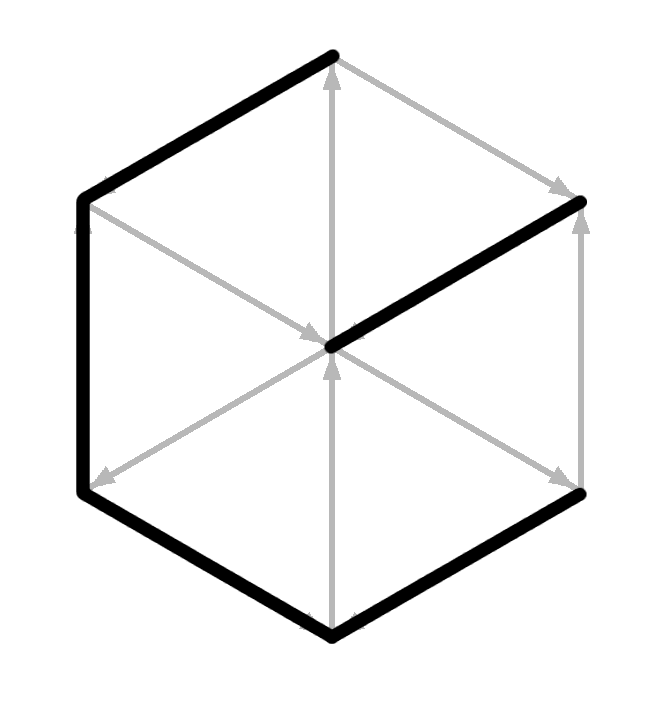}
\includegraphics[scale=0.090]{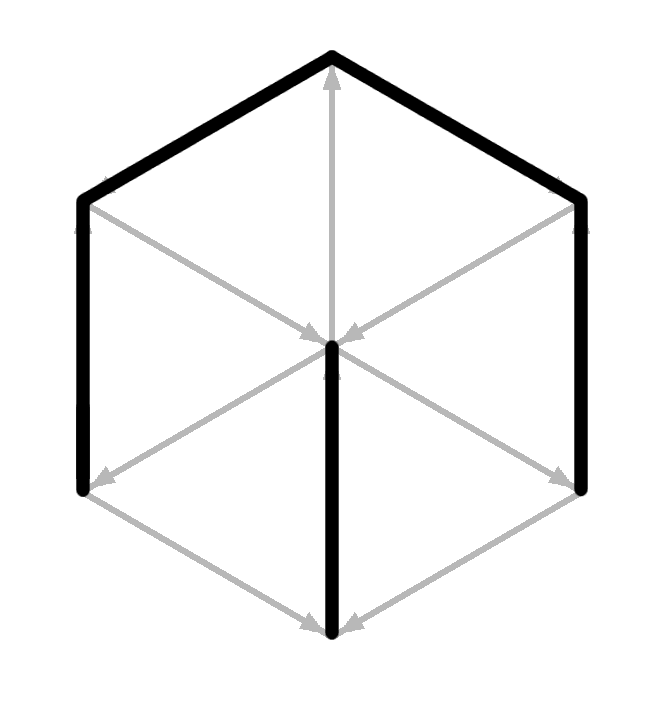}
\includegraphics[scale=0.090]{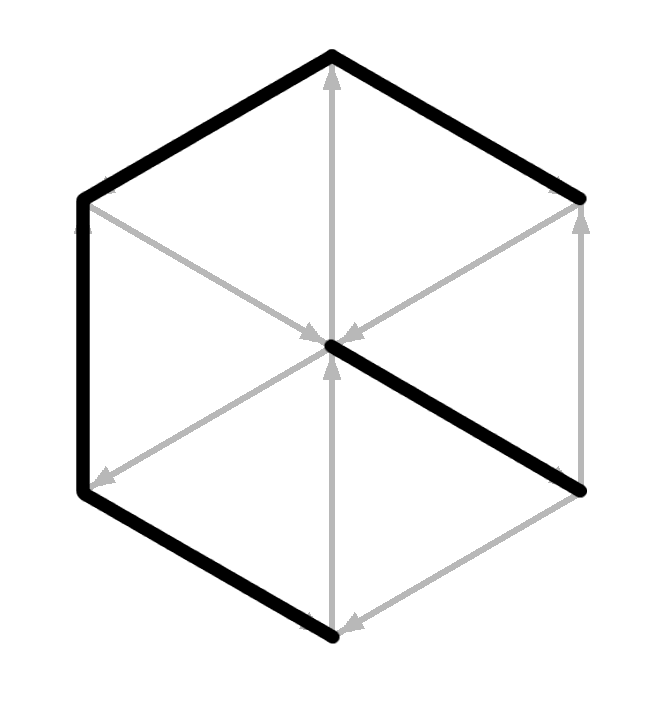}
\includegraphics[scale=0.090]{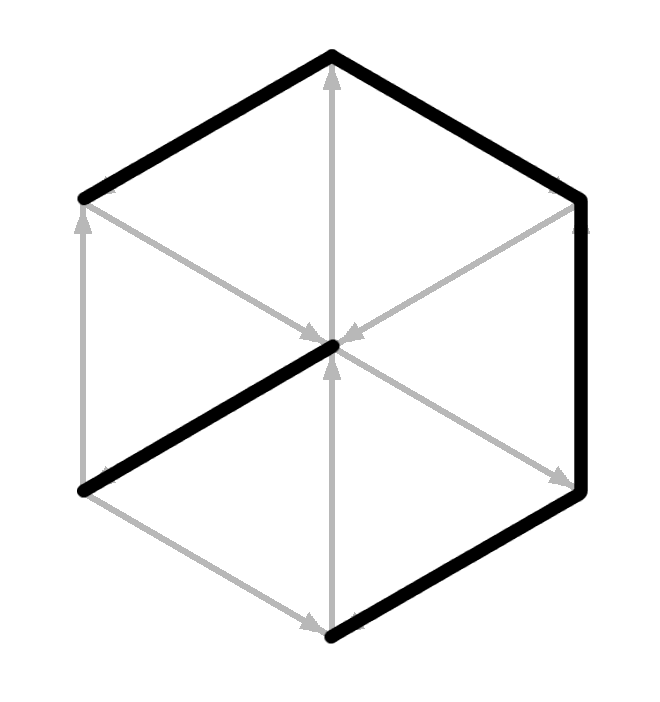}
\includegraphics[scale=0.090]{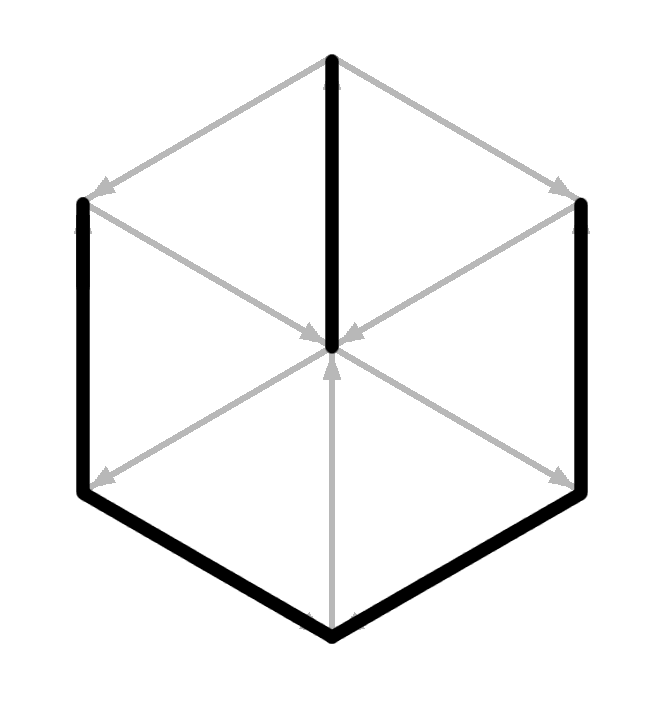}
\caption{\label{figure-12} The $x$-, $y$- and $z$-
$(1,2)$-sources, the $x$-, $y$- and $z$- $(2,1)$-sources.}
\end{center} \begin{center}
\includegraphics[scale=0.090]{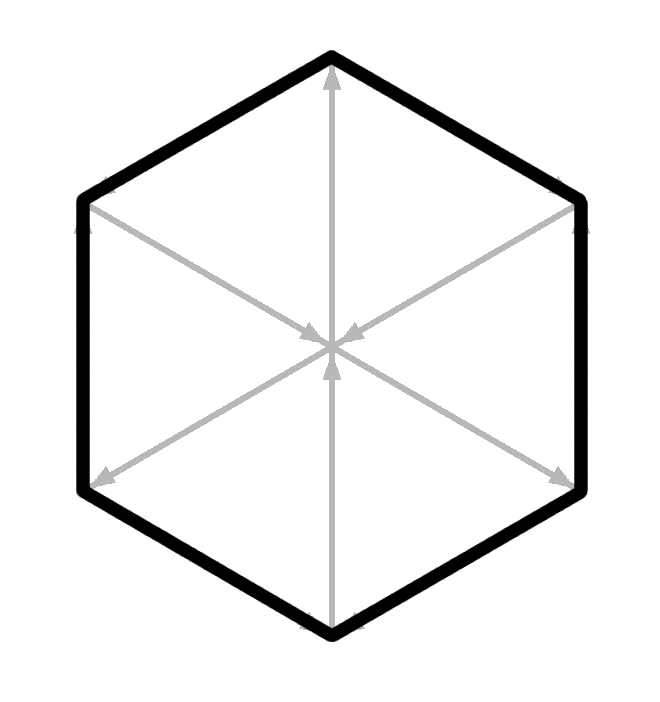}
\caption{\label{figure-13} The $(3,3)$-source} \end{center}
\end{figure} 

The \it sink-source graph \rm $SS_{\mathcal{M}, E}$ is a graph drawn 
on top of $Q(G)$ whose vertices are the sinks and the sources and whose 
edges are the charge lines. The torus $T_G$ is then 
divided up by $SS_{\mathcal{M}, E}$ into several regions with 
all vertices interior to any one region being tiles of same orientation. 
It turns out 
(\cite{CautisLogvinenko-ADerivedApproachToGeometricMcKayCorrespondence},
Proposition $4.14$ and then as per Corollary $4.16$) that there are
only three possible shapes that $SS_{\mathcal{M}, E}$ can take and we depict 
them on Figures \ref{figure-14a} - \ref{figure-14c} indicating by thin 
grey lines the orientation of tiles in that region. In case of Figure 
$\ref{figure-14b}$
the rotations of the depicted shape by $\frac{2\pi}{3}$ and 
$\frac{4\pi}{3}$ are also possible and since one of the tile regions
is non-contractible we indicated by a dotted line the fundamental
domain of $Q(G)$. \begin{figure}[h] \centering 
\subfigure[A single $(3,3)$-source] { \label{figure-14a}
\includegraphics[scale=0.11]{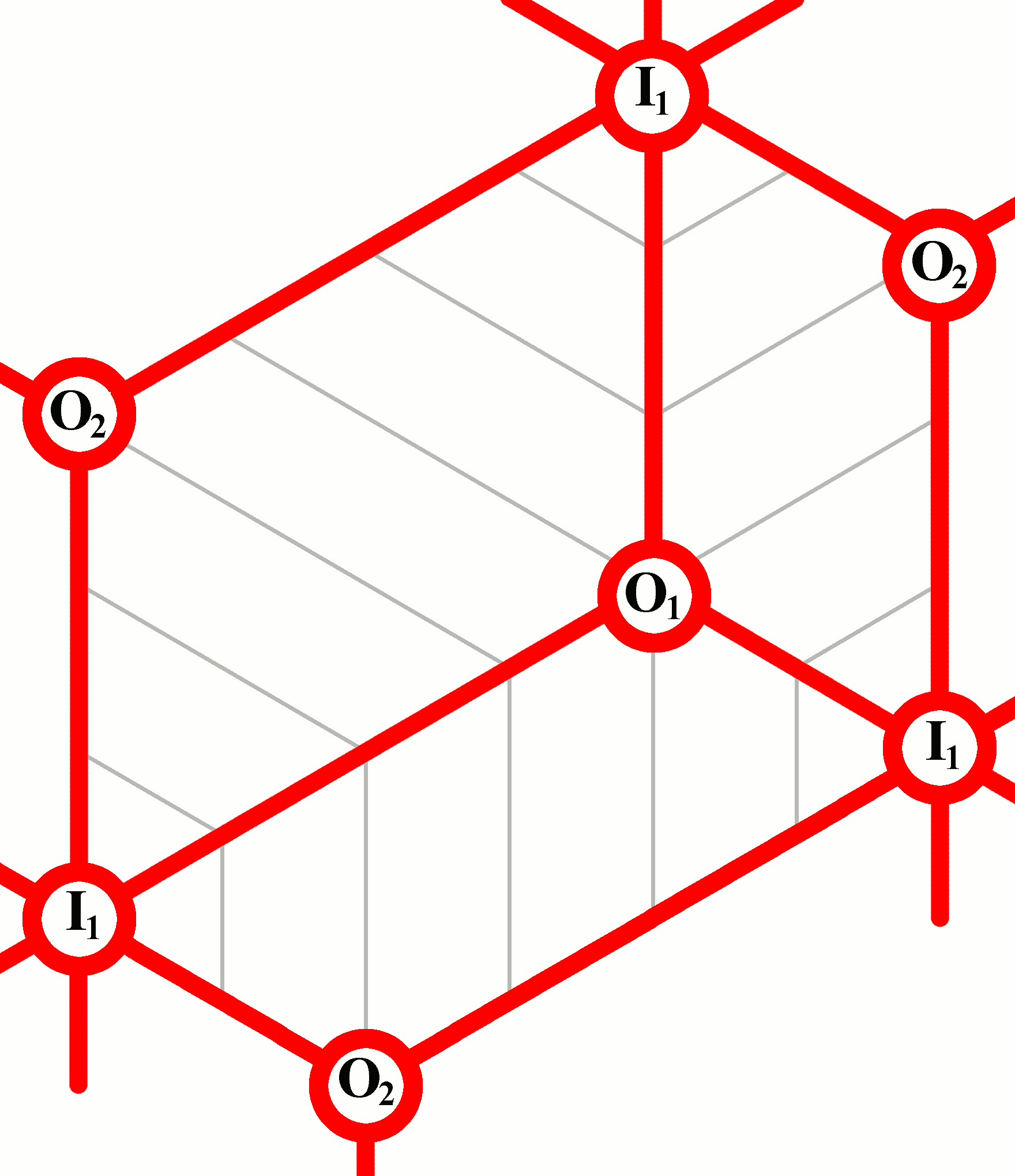} } 
\subfigure[One $(1,2)$-source and one
$(2,1)$-source] { \label{figure-14b}
\includegraphics[scale=0.12]{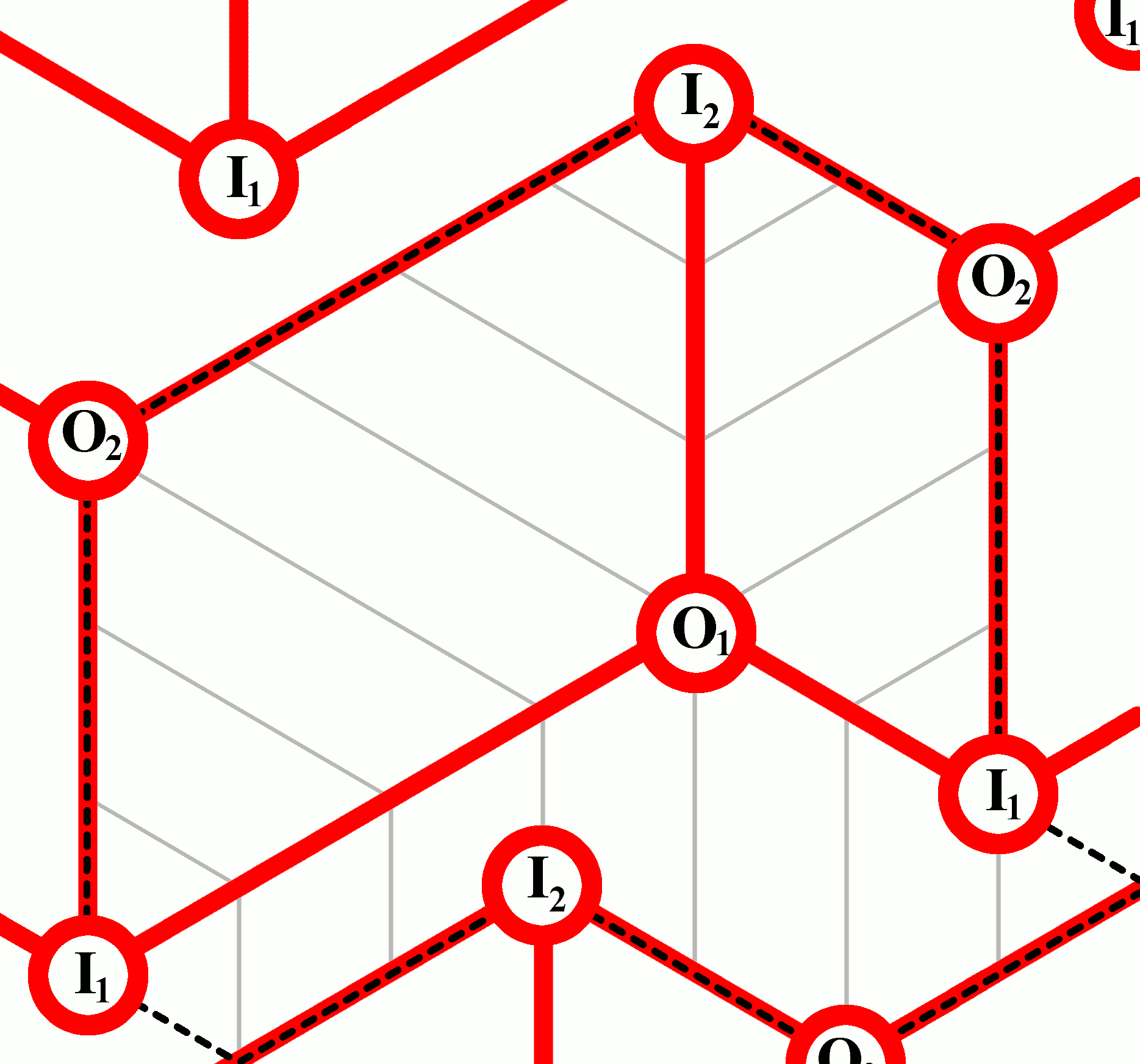} } 
\subfigure[Three $(2,1)$-sources] { \label{figure-14c}
\includegraphics[scale=0.12]{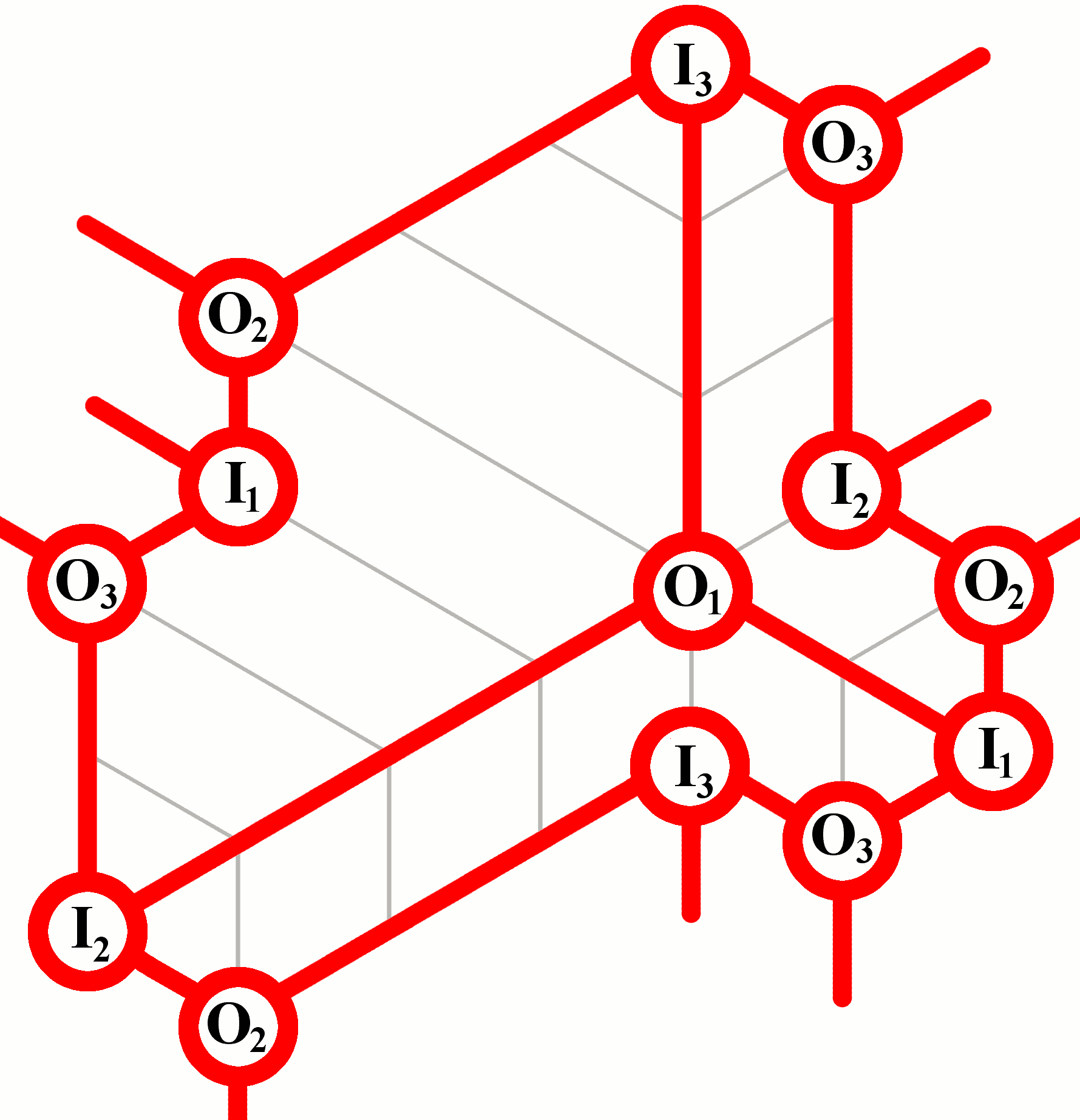} } \caption{Sink-source
graphs in the case of a single $(0,3)$-sink} \label{figure-14}
\end{figure}

It is worth noting that sink-source graph $SS_{\mathcal{M}, E}$ 
determines completely the divisor $E$ which gave rise to it. This
is because it completely determines which arrows of
$Q(G)$ do and which do not vanish along $E$: no arrow 
which belongs to one of the charge lines vanishes along $E$ and within each 
of the regions into which $SS_{\mathcal{M}, E}$ divides up $T_G$ the
arrows that vanish are those that have the same orientation as the tile
vertices of that region. Then the following result can be applied to 
determine the toric coordinates of $e \in \mathfrak{E}$ which 
corresponds to $E$:
\begin{lemma} \label{lemma-number-of-diamonds}
Let $e \in \mathfrak{E}$. Suppose the total number of $x$-oriented arrows
of $Q(G)$ which vanish along $E_e$ in $Q(G)_{\mathcal{M}}$ is $a$, 
the total number of $y$-oriented arrows 
is $b$ and the total number of $z$-oriented arrows is $c$. 
Then $e = \frac{1}{|G|}(a,b,c) \in L \subset \mathbb{Q}^3$.
\end{lemma}
\begin{proof}
Consider the sum of the vanishing divisors of all $x$-oriented arrows 
of $Q(G)$: 
\begin{align}\label{eqn-sum-of-x-arrows}
\sum_{\chi \in G^\vee} B_{\chi,x}
\end{align}
The divisor $E_e$ appears in each $B_{\chi,x}$ with multiplicity 
$1$ if the arrow $(\chi,x)$ vanishes along $E_e$ and with 
multiplicity $0$ if it doesn't. We conclude that the multiplicity 
with which $E_e$ occurs in $\eqref{eqn-sum-of-x-arrows}$ is $a$, 
the number of the $x$-oriented arrows of $Q(G)$ which
vanish along $E_e$. 

On the other hand, write $\mathcal{M}$ as $\bigoplus_{\chi \in G^\vee}
\mathcal{L}(-D_\chi)$ (see Section
\ref{section-G-clusters-and-g-graphs}). Then 
$B_{\chi,x} = D_{\chi^{-1}} + (x) -
D_{\kappa(x)^{-1} \chi^{-1}}$, where $(x) = \sum_{f \in \mathfrak{E}}
\text{val}_{E_f}(x)$ is the principal $G$-Weil divisor of $x$
(see \cite{Logvinenko-DerivedMcKayCorrespondenceViaPureSheafTransforms},
Section 4.3). So we re-write $\eqref{eqn-sum-of-x-arrows}$ as
\begin{align}\label{eqn-sum-of-x-arrows-2}
\sum_{\chi \in G^\vee} \left( 
D_{\chi^-1} + (x) - D_{\kappa(x)^{-1} \chi^{-1}}
\right).
\end{align}
All $D_\chi$ in the sum cancel out and we are left with $|G| (x)$. 
Therefore the multiplicity of $E_e$ in \eqref{eqn-sum-of-x-arrows-2}
is $|G|\text{val}_{E_e}(x) = \text{val}_{E_e}(x^{|G|}) = e(|G|,0,0) = 
|G| e(1,0,0)$. We conclude that $e(1,0,0) = \frac{1}{|G|} a$. 
Similarly $e(0,1,0) = \frac{1}{|G|} b$ and $e(0,0,1) = \frac{1}{|G|} c$. 
The result follows.
\end{proof}

Most of the results on sink-source graphs in
\cite{CautisLogvinenko-ADerivedApproachToGeometricMcKayCorrespondence}
are stated for the dual $\tilde{\mathcal{M}}$ of the universal family
$\mathcal{M}$ of $G$-clusters, as it is $\tilde{\mathcal{M}}$ that is used
for computing the transforms $\Psi(\mathcal{O}_0 \otimes \chi)$. It
is more convenient for us to work with $\mathcal{M}$ in the present paper, 
so the following general lemma is useful: 
\begin{lemma} \label{class-in-a-family-from-a-class-in-its-dual}
Let $\mathcal{F}$ be any $\gnat$-family on $Y$, let $E \in
\Except(Y)$ and let $\chi \in G^\vee$. 
\begin{enumerate}
\item $\chi$ is an $(a,b)$-source (resp. sink) in
$SS_{\mathcal{F}, E}$ if and 
only if $\chi^{-1}$ is a $(b,a)$-source (resp. sink) in
$SS_{\tilde{\mathcal{F}}, E}$.
\item $\chi$ is an $x$-$(a,b)$-charge in $SS_{\mathcal{F}, E}$ if and 
only if $\chi^{-1}$ is an $x$-$(b,a)$-charge in
$SS_{\tilde{\mathcal{F}}, E}$. Similarly for $y$- and $z$-oriented charges.
\item $\chi$ is an $x$-tile in $SS_{\mathcal{F}, E}$ if and only if 
$\chi^{-1}$ is an $x$-tile in $SS_{\tilde{\mathcal{F}}, E}$. 
Similarly for $y$- and $z$-oriented tiles. 
\end{enumerate}
\end{lemma}
\begin{proof}
From the definition of the dual of a $\gnat$-family in 
\cite{CautisLogvinenko-ADerivedApproachToGeometricMcKayCorrespondence},
Section $2.2$ it follows that for any $\xi \in G^\vee$
the $\xi$-eigensheaf $\tilde{\mathcal{F}}_\xi$ of $\tilde{\mathcal{F}}$ 
is precisely the dual of the $\xi^{-1}$-eigensheaf $\mathcal{F}_{\xi^{-1}}$ of
$\mathcal{F}$. It further folows that the map $\alpha'_{\xi,x}$
which corresponds to the arrow $(\xi,x)$ 
in the associated representation $Q(G)_{\tilde{\mathcal{F}}}$ is precisely
the dual of the map $\alpha_{\xi^{-1}\kappa(x)^{-1}, x}$ in
$Q(G)_{\mathcal{F}}$. Therefore $\alpha'_{\xi,x}$ vanishes along $E$
if and only if $\alpha_{\xi^{-1}\kappa(x)^{-1}, x}$ vanishes along $E$. 
Applying this to every arrow in the subquiver $\hex(\chi)$ 
surrounding $\chi$ for every case depicted on 
Figures \ref{figure-09} - \ref{figure-13} yields the claim. 
\end{proof}

The importance of sink-source graphs for us lies in the fact
that the sinks and the sources of $SS_{\mathcal{M}, E}$ are precisely
the characters $\chi \in G^\vee$ for which $E \subset
\supp(\Psi(\mathcal{O}_0 \otimes \chi))$:
\begin{prps} \label{sinks-and-sources-iff-in-transforms-support}
Let $E \in \Except(Y)$ and $\chi \in G^\vee$. Then:
\begin{enumerate}
\item $E \subset \supp H^0(\Psi(\mathcal{O}_0 \otimes \chi))$
if and only if $\chi$ is a $(3,0)$-sink in $SS_{\mathcal{M}, E}$. 
\item $E \subset \supp H^{-1}(\Psi(\mathcal{O}_0 \otimes \chi))$
if and only if $\chi$ is a source in $SS_{\mathcal{M}, E}$. 
\item $E \subset \supp H^{-2}(\Psi(\mathcal{O}_0 \otimes \chi))$
if and only if $\chi$ is a $(0,3)$-sink in $SS_{\mathcal{M}, E}$, 
that is if $\chi$ is the trivial character $\chi_0$.
\end{enumerate}
\end{prps}
\begin{proof}
By
\cite{CautisLogvinenko-ADerivedApproachToGeometricMcKayCorrespondence},
Proposition 4.6, the transform $\Psi(\mathcal{O}_0 \otimes \chi)$
is given by the total complex of the skew-commutative cube of
line-bundles induced by the subrepresentation
$\hex(\chi^{-1})_{\tilde{\mathcal{M}}}$ of the associated
representation $Q(G)_{\tilde{\mathcal{M}}}$. Then 
\cite{CautisLogvinenko-ADerivedApproachToGeometricMcKayCorrespondence},
Lemma 3.1 reduces the question of whether $E \subset \supp
H^{-i}(\Psi(\mathcal{O}_0 \otimes \chi))$ for $i = 0$, $-1$ or $-2$ to  
knowing which of the maps in the cube vanish along $E$, i.e. which 
of the arrows in $\hex(\chi^{-1})_{\tilde{\mathcal{M}}}$ vanish along
$E$. We can therefore go through every vertex class depicted on Figures 
$\ref{figure-09}$-$\ref{figure-13}$ and apply
\cite{CautisLogvinenko-ADerivedApproachToGeometricMcKayCorrespondence}, 
Lemma 3.1 to see whether $E \subset 
\supp H^{-i}(\Psi(\mathcal{O}_0 \otimes \chi))$ for some $i$ 
if $\chi^{-1}$ is of that class in $SS_{\tilde{\mathcal{M}}, E}$. 
Finally, we translate the class of $\chi^{-1}$ in
$SS_{\tilde{\mathcal{M}}, E}$
to the class of $\chi$ in $SS_{\mathcal{M}, E}$ via Lemma
\ref{class-in-a-family-from-a-class-in-its-dual}. 

For example, by \cite{CautisLogvinenko-ADerivedApproachToGeometricMcKayCorrespondence},
Lemma 3.1 we have $E \subset \supp H^{0}(\Psi(\mathcal{O}_0 \otimes
\chi))$ if and only if $E$ belongs to the vanishing divisors of
maps $\alpha^1$, $\alpha^2$ and $\alpha^3$ of the cube, which 
translates to arrows $(\chi^{-1} \kappa(x)^{-1}, x)$, 
$(\chi^{-1} \kappa(y)^{-1}, y)$ and $(\chi^{-1} \kappa(z)^{-1}, z)$
in $\hex(\chi^{-1})_{\tilde{\mathcal{M}}}$. From 
Figures $\ref{figure-09}$-$\ref{figure-13}$
this can only happen when $\chi^{-1}$ is a $(0,3)$-sink in
$SS_{\tilde{\mathcal{M}}, E}$, i.e. $\chi$ is a $(3,0)$-sink
in $SS_{\mathcal{M}, E}$. 
\end{proof}

\section{Main results}

\begin{prps}\label{prps-(3,3)-source-graph-to-case-1}
Let $e \in \mathfrak{E}$ be such that $E_e \subset \Except(Y)$. If 
the graph $SS_{\mathcal{M}, E_e}$ is as depicted on Figure 
\ref{figure-15a} then the vertex $e$ in the triangulation $\Sigma$ 
looks locally as depicted on Figure $\ref{figure-15b}$. 
The monomial ratios carving out the edges incident to $e$
can be computed in terms of the indicated lengths in 
$SS_{\mathcal{M}, E_e}$ as shown on Figure $\ref{figure-15b}$. 
And $e = \frac{1}{|G|}(bc, ac, ab)$ in $L \subset \mathbb{Q}^3$.
\end{prps}
\begin{figure}[h] \centering 
\subfigure[A one $(3,3)$-source graph] { \label{figure-15a} 
\includegraphics[scale=0.13]{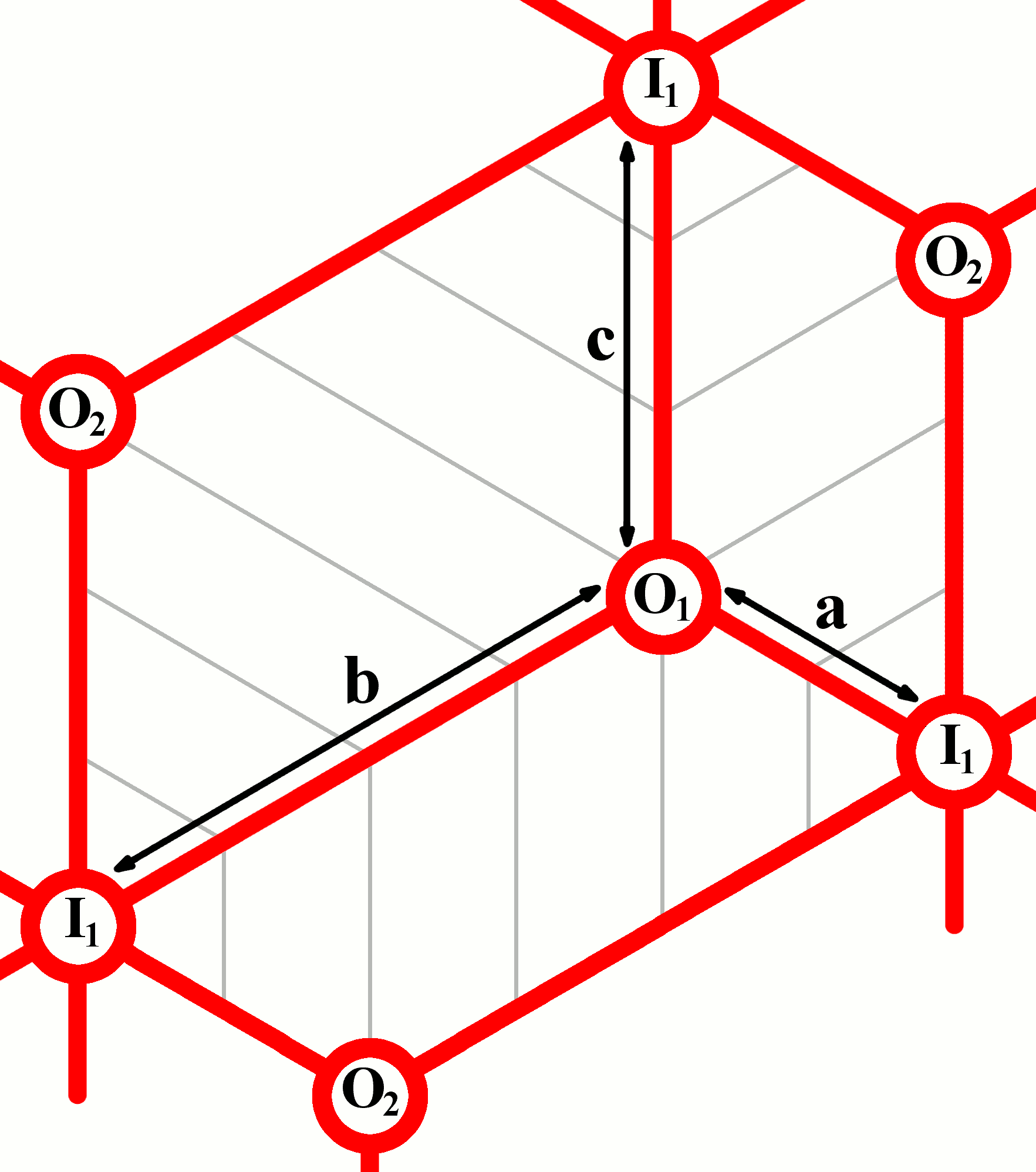}} \hspace{0.1cm}
\subfigure[Case $\ref{case-meeting-of-champions}$ of Prps.
$\ref{prps-the-classification-of-the-vertices-of-the-triangulation}$] { \label{figure-15b}
\includegraphics[scale=0.32]{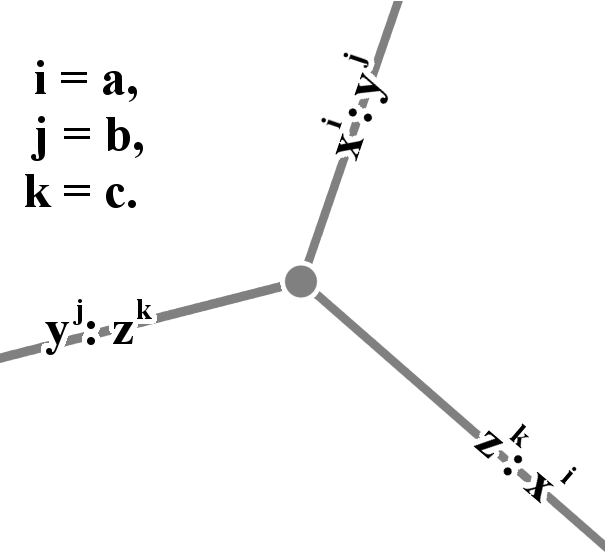}} 
\caption{The correspondence of Proposition
\ref{prps-(3,3)-source-graph-to-case-1}} \label{figure-15}
\end{figure}
\begin{proof}
We proceed by showing how the shape of the sink-source graph 
$SS_{\mathcal{M}, E_e}$ imposes restrictions on which monomial ratios 
can mark the edges which are incident to the vertex $e$ in the triangulation
$\Sigma$. 

Suppose there is an edge incident to $e$ which is carved out by 
a ratio $x^{i'} : y^{j'} z^{k'}$ for some $i',j',k' \neq 0$. 
Let $\sigma$ and $\sigma'$ be the two
triangles containing the edge in question.  Then by Corollary 
$\ref{cor-markings-belong-to-ggraphs}$ one of $x^{i'}$ and $y^{j'} z^{k'}$
must belong to the $G$-graph $\Gamma_\sigma$ of the toric fixed 
point $E_{\sigma}$ and the other to $\Gamma_{\sigma'}$. Without loss of
generality, assumed $x^{i'} \in \Gamma_\sigma$ and $y^{j'} z^{k'} \in
\Gamma_{\sigma'}$. Then, by definition 
of a $G$-graph, $x^{i'}$ doesn't belong to the ideal defining the
$G$-cluster $\mathcal{M}_{|E_{\sigma}}$. Therefore $x^{i'} \cdot 1
\neq 0$ in $\mathcal{M}_{|E_{\sigma}}$ and, similarly, $y^{j'} z^{k'}
\cdot 1 \neq 0$ in $\mathcal{M}_{|E_{\sigma'}}$. Translating this
into the language of the associated representation $Q(G)_{\mathcal{M}}$
we see that in the path of $i'$ $x$-oriented arrows
in $Q(G)$ which starts at $\chi_0$ no arrow can vanish at $E_{\sigma}$.
Similarly, in any path which starts at $\chi_0$ and consists of $j'$
$y$-oriented arrows and $k'$ $z$-oriented arrows no arrow can vanish
at $E_{\sigma'}$. But the vanishing locus of any arrow in
$Q(G)_{\mathcal{M}}$ is a divisor of form $\sum_{f \in \mathfrak{E}}
b_{f} E_f$ where $b_f = 1$ or $0$. Therefore, if an arrow doesn't
vanish at $E_{\sigma}$ it doesn't vanish at any $E_{f}$ such that
$E_{\sigma} \in E_{f}$. In other words, it doesn't vanish at any of
the three divisors corresponding to the three vertices of the triangle
$\sigma$ in $\mathfrak{E}$. In particular, it doesn't vanish at
$E_{e}$. Similarly if an arrow doesn't vanish at $E_\sigma'$
it also doesn't vanish along $E_e$.  

Denote by $X$ the vertex of $Q(G)$ corresponding to the common 
character of $x^{i'}$ and $y^{j'} z^{k'}$. Recall that on 
Figure \ref{figure-15a} the vertex $O_1$ corresponds to the trivial
character $\chi_0$.  So we have a path of $i'$ $x$-oriented arrows
which begins at $O_1$, terminates at $X$ and none of the arrows vanish 
along $E_e$. It must not 
therefore enter the region tiled with $x$-oriented tiles as within 
this region every $x$-oriented arrow vanishes at $E_e$. It is evident 
from Figure \ref{figure-15a} that the whole path, together with its
endpoint $X$, must therefore be contained in $x$-$(0,1)$-charge line 
$O_1 I_1$ and the $x$-$(1,0)$-charge line $I_1 O_2$ which follows upon
it. On the other hand, the paths which start at $O_1$ terminate at 
$X$, and consist of $j'$ $y$-oriented arrows and $k'$ $z$-oriented
arrows, sweep out a parallelogram betwen $O_1$ and $X$ with sides 
of length $j'$ and $k'$. As by assumption neither $j' = 0$ nor $k' = 0$
this parallelogram is non-degenerate. As no $y$- or $z$-oriented arrow 
within this parallelogram vanishes along $E_e$ the whole parallelogram must be
contained within the region tiled with $x$-oriented tiles. In
particular, $X$ itself must lie within this region or on its boundary.
From Figure \ref{figure-15a} it is evident that the only 
common points of the $x$-arrow path $O_1 I_1 O_2$ on which $X$ must lie
and of the $x$-tiled region are $O_1$, $I_1$ and $O_2$. We can't have $X =
O_1$ or $I_1$, as then the parallelogram would be degenerate which contradicts
$j',k' \neq 0$. We conclude that $X = O_2$, the parallelogram is 
the whole of the $x$-tiled region and $(i',j',k') = (2a,b,c)$. 
The ratio would then be $x^{2a} : y^{b} z^{c}$. But this is
impossible since $\frac{x^{2a}}{y^{b} z^{c}}$ decomposes as 
$\frac{x^{a}}{y^b}\;\frac{x^a}{z^c}$ and so can't carve
out an edge of $\Sigma$ by Corollary $\ref{cor-primitivity-of-marking-ratios}$. 
Repeating the same argument 
for the ratios of form $y^{j'} : z^{k'} x^{i'}$ and $z^{k'} : x^{i'} y^{j'}$ 
we see that neither of them can occur either. 

Suppose now there is an edge incident to $e$ which is carved out by 
the ratio $x^{i'} : y^{j'}$ for some $i',j' \neq 0$. Denote by $X$ the vertex
 which corresponds to the common character of $x^{i'}$ and
$y^{j'}$. Arguing as above we see that $X$ must lie both somewhere on 
the $x$-arrow path $O_1 I_1 O_2$ and 
somewhere on the $y$-arrow path $O_1 I_1 O_2$. From Figure \ref{figure-15a} 
it is evident that $X$ must then be either $O_1$, $I_1$ or $O_2$. 
If $X = O_1$, then $i' = j' = 0$ which contradicts our assumption. 
If $X = O_2$ then $i' = 2a$ and $j' = 2b$, so the ratio marking the incident 
edge would be $x^{2a} : y^{2b}$. This contradicts the fact that by its
definition any ratio marking an edge must come from a primitive
Laurent monomial.  Therefore we must have $X = I_1$ and then $(i',j')
= (a,b)$ and the ratio is $x^a : y^b$. Repeating the same argument
for ratios of form $y^{j'} : z^{k'}$ and $z^{k'} : x^{i'}$ we
conclude that the only monomial ratios which can mark the edges
incident to $e$ in $\Sigma$ are:
\begin{align}
x^a : y^b, \quad y^b : z^c,  \quad z^c : x^a.
\end{align}
Consulting the classification of the
Proposition
$\ref{prps-the-classification-of-the-vertices-of-the-triangulation}$ 
we see that $e$ must necessarily belong to Case
$\ref{case-meeting-of-champions}$ reproduced on Figure
$\ref{figure-15b}$, and that we must have $i = a$, $j = b$ and $k=c$.

Finally, to obtain $e = \frac{1}{|G|}(bc, ac, ab)$ we apply
Lemma \ref{lemma-number-of-diamonds}. Observe that the number 
of $x$-oriented arrows which vanish along $E_e$
is precisely the number of $x$-arrows in  
the $x$-tiled region. We count the latter by breaking up 
the $x$-tiled region into little parallelograms 
whose sides are a single $y$-arrow and a single $z$-arrow. 
Each such parallelogram contains exactly one $x$-oriented
arrow as its diagonal and on Figure $\ref{figure-15a}$ we
see the $x$-tiled region consists of $bc$ such parallelograms, since
the $x$-tiled region is itself a big parallelogram with two sides
consisting one of $b$ $y$-arrows and the other of $c$
$z$-arrows. Similarly, we see see that the number of
$y$-arrows which vanish along $E_e$ is $ac$ and the number 
of $z$-arrows is $ab$. Therefore by Lemma \ref{lemma-number-of-diamonds}
we have $e = \frac{1}{|G|}(bc, ac, ab)$. 
\end{proof}

\begin{figure}[!htb] \centering 
\subfigure[A one $(1,2)$-source and one $(2,1)$-source graph]{\label{figure-16a} 
\includegraphics[scale=0.11]{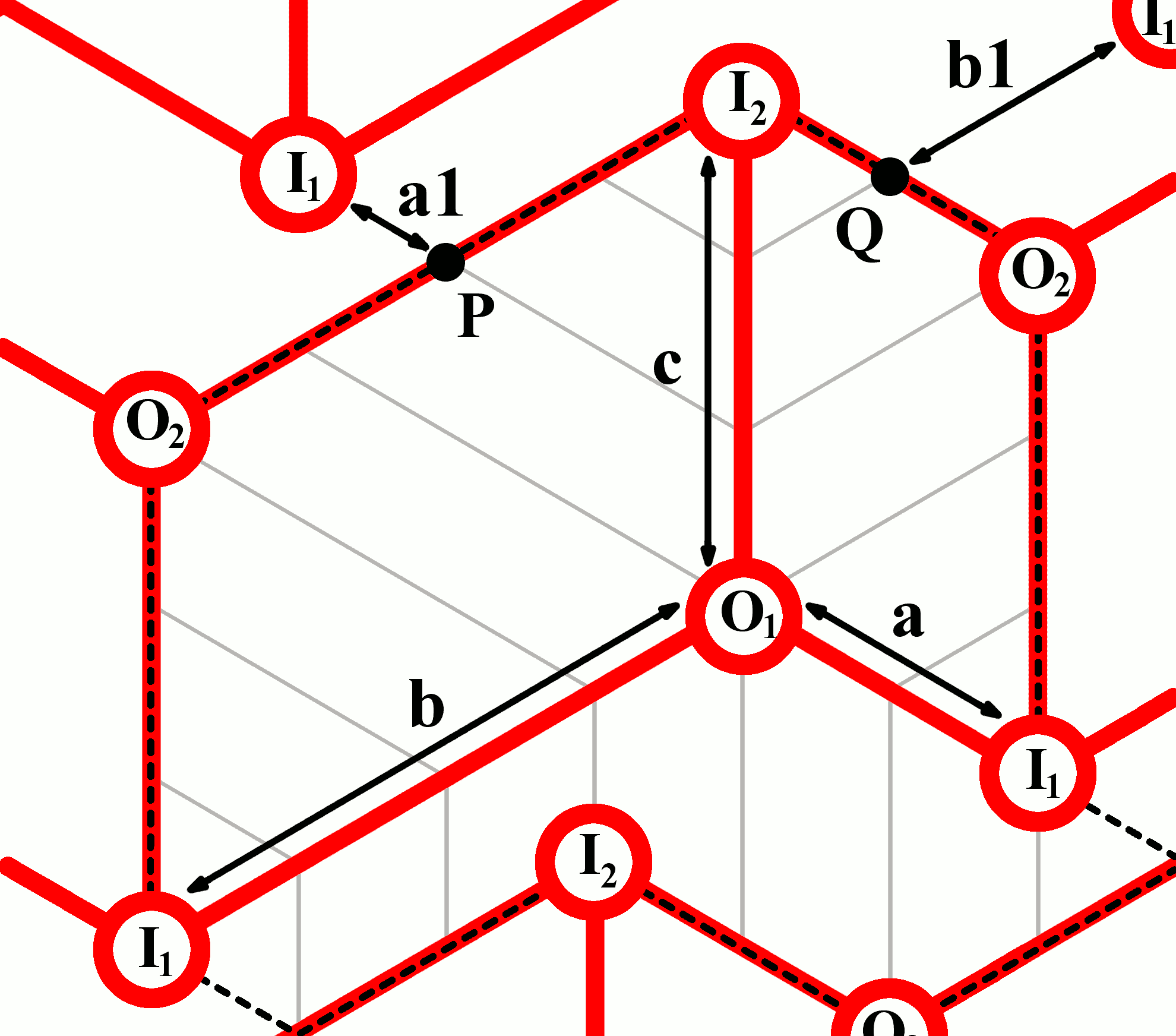}}  
\subfigure[Case $\ref{case-one-line-from-vertex}(a)$ of Prps.
$\ref{prps-the-classification-of-the-vertices-of-the-triangulation}$] { \label{figure-16b}
\includegraphics[scale=0.29]{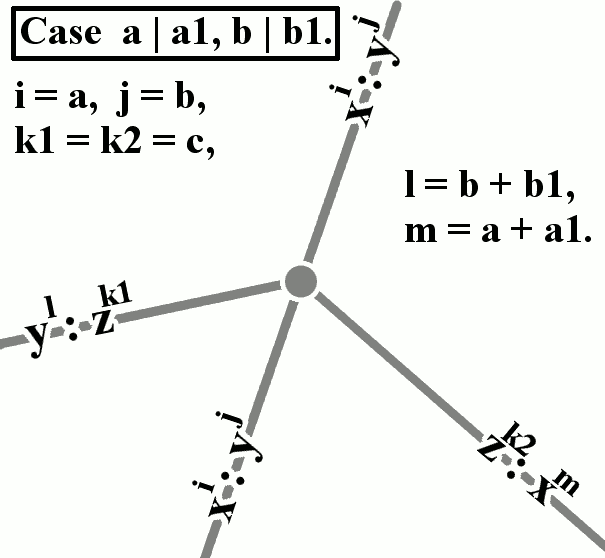}} \\
\subfigure[Case $\ref{case-one-line-from-vertex}(b)$ of Prps.
$\ref{prps-the-classification-of-the-vertices-of-the-triangulation}$] { \label{figure-16c}
\includegraphics[scale=0.29]{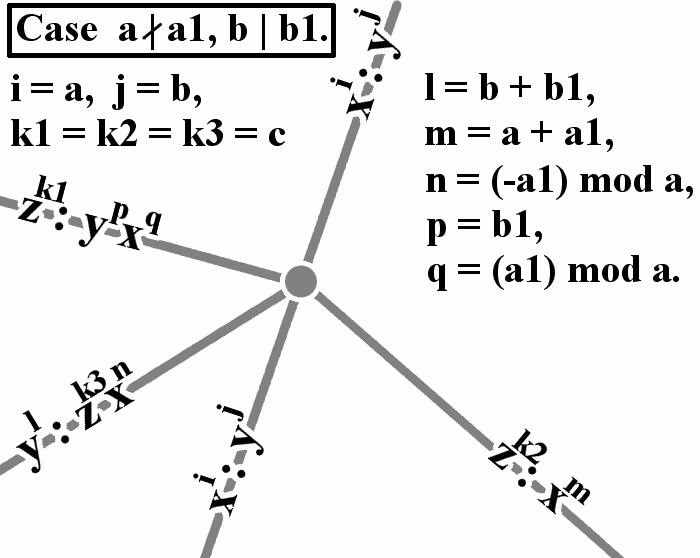}} 
\subfigure[Case $\ref{case-one-line-from-vertex}(b)$ of Prps.
$\ref{prps-the-classification-of-the-vertices-of-the-triangulation}$, 
with $x$ and $y$ permuted] { \label{figure-16d}
\includegraphics[scale=0.29]{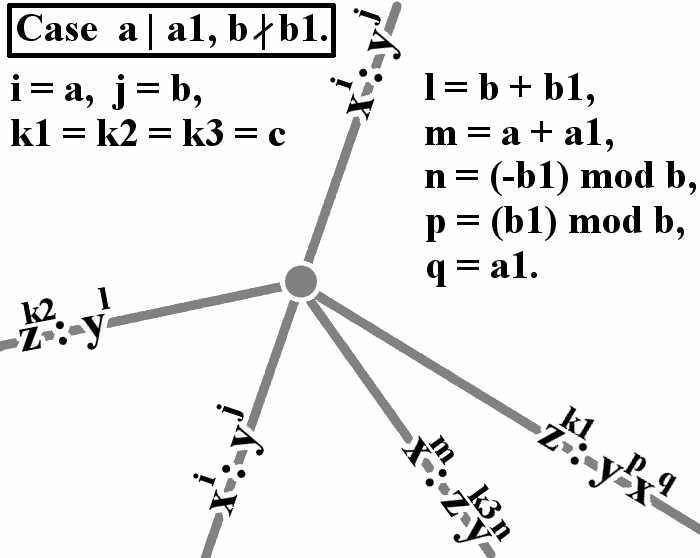}} \\
\subfigure[Case $\ref{case-one-line-from-vertex}(c)$ of Prps.
$\ref{prps-the-classification-of-the-vertices-of-the-triangulation}$] { \label{figure-16e}
\includegraphics[scale=0.29]{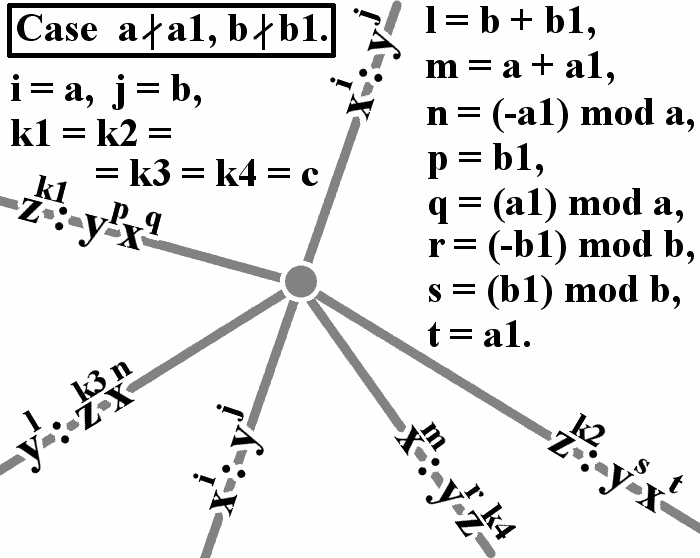}} 
\caption{The correspondence of Proposition
\ref{prps-one-(2,1)-source-one-(1,2)-source-graph-to-case-2}} \label{figure-16}
\end{figure}
\begin{prps}\label{prps-one-(2,1)-source-one-(1,2)-source-graph-to-case-2}
Let $e \in \mathfrak{E}$ be such that $E_e \in \Except(Y)$. If 
the graph $SS_{\mathcal{M}, E_e}$ is as depicted on Figure 
\ref{figure-16a} then the vertex $e$ in 
the triangulation $\Sigma$ looks locally as depicted on:
\begin{enumerate}
\item Figure $\ref{figure-16b}$ if $a \mid a1$ and $b \mid b1$.
\item Figure $\ref{figure-16c}$ if $a \mid a1$ and $b \nmid b1$.
\item Figure $\ref{figure-16d}$ if $a \nmid a1$ and $b \mid b1$.
\item Figure $\ref{figure-16e}$ if $a \nmid a1$ and $b \nmid b1$.
\end{enumerate}
The monomial ratios carving out the edges incident to $e$ can be
computed in terms of the indicated lengths in $SS_{\mathcal{M}, E_e}$
as shown on Figures $\ref{figure-16b}$-$\ref{figure-16d}$. 
If the shape of $SS_{\mathcal{M}, E_e}$ is a rotation of Figure
\ref{figure-16a} by $\frac{2\pi}{3}$ or $\frac{4\pi}{3}$ one permutes
$x$, $y$ and $z$ in Figures $\ref{figure-16b}$-$\ref{figure-16d}$
accordingly. The coordinates of $e$ in $L$ are 
$\frac{1}{|G|}(bc, ac, |G| - bc - ac)$. 
\end{prps}
\begin{proof}
Assume that $SS_{\mathcal{M}, E_e}$ is exactly as depicted on Figure 
$\ref{figure-16a}$. The cases of rotation of Figure $\ref{figure-16a}$
by $\frac{2\pi}{3}$ or $\frac{4\pi}{3}$ are dealt with in exactly
the same way.

We employ the same method here as for Proposition 
$\ref{prps-(3,3)-source-graph-to-case-1}$. However this time 
the $z$-tiled region is non-contractible, wrapping around the torus. 
We can perfectly well have two straight lines intersecting at more than
one point within it. This gives rise to some minor technical difficulties,
so we need to establish several auxiliary facts. Consider
a path of $x$-oriented arrows which begins at $O_1$. It first
travels along the $x$-$(0,1)$-charge line $O_1 I_1$. Past the vertex $I_1$ 
the path enters the $z$-tiled region and travels within it until 
it encounters the $y$-$(1,0)$-charge line $O_2 I_2$. We define
$P$ to be the vertex where it happens and we define $a_1$ be
the length of the $x$-arrow path $I_1 P$. Similarly, we define 
$Q$ to be the point where the $y$-arrow path which starts at $O_1$ 
first meets the $x$-$(1,0)$-charge line $I_2 O_2$ and $b_1$ 
to be the length of the $y$-arrow path $I_1 Q$. As noted above 
it is perfectly possible for the paths $I_1 P$ and $I_1 Q$ to 
intersect several times within the $z$-tiled region (even though
it doesn't happen on Figure $\ref{figure-16a}$). Let $C$ be
any vertex where they intersect. Let $a_c$ be the length of 
the $x$-arrow path $I_1 C$ and $b_c$ be the length of $y$-arrow 
path $I_1 C$. We claim that $a_c = n' a$ and $b_c = n' b$ for 
some integer $n' \in \mathbb{Z}$. This is because there
exists a natural isomorphism $M / \mathbb{Z}(1,1,1) \xrightarrow{\sim}
H_1(T_G,\mathbb{Z})$ where $M$ is the lattice of 
$G$-invariant Laurent monomials (see Section
\ref{section-ghilb-and-toric}). This isomorphism sends $[i',j',k']$, 
the class of the monomial $x^{i'} y^{j'} z^{k'}$, to the class of 
a loop in $T_G$ consisting of $i'$ $x$-arrows, 
$j'$ $y$-arrows and $k'$ $z$-arrows 
(\cite{Logvinenko-thesis}, Lemma 6.41).
Now $z$-tiled region is clearly contractible to a $1$-sphere, 
so its first homology is $\mathbb{Z}$. On Figure $\ref{figure-16a}$ 
we see that the loop $O_1 I_1 O_1$ formed by 
$x$-$(0,1)$- and $y$-$(0,1)$-charge lines lies within the $z$-tiled
region and wraps around it exactly once. Therefore its
class must generate the first homology of the $z$-tiled region. 
This class is $[\frac{x^a}{y^b}]$ and therefore the class of every loop 
contained within the $z$-tiled region must be a multiple of it. 
In particular $[\frac{x^{a_c}}{y^{b_c}}]$, which is the class of the loop 
$I_1 C I_1$. The claim follows. 

Similarly, denote by $a_2$ the length of the $x$-arrow path $I_2 Q$
and by $b_2$ the length of the $y$-arrow path $I_2 P$. Consider
the loop $I_1 P I_2 Q I_1$. Its homology class is $[\frac{x^{a_1 +
a_2}}{y^{b_1 + b_2}}]$. As this loop lies within the $z$-tiled 
region its class must be a multiple of $[\frac{x^a}{y^b}]$. 
As clearly $0 \leq a_2 < a$ and $0 \leq b_2 < b$ we conclude 
that $a_2 = (- a_1) \mod a$ and $b_2 = (- b_1) \mod b$.  

We can now proceed to the main proof. Suppose there is an edge incident to 
$e$ which is carved out by a ratio $x^{i'} : y^{j'} z^{k'}$ for some 
$i',j',k' \neq 0$. Denote by $X$ the vertex of $Q(G)$
corresponding to the common character of $x^{i'}$ and $y^{j'} z^{k'}$.
Assume first that $b \nmid b_1$ and therefore $P \neq I_2$. Then 
travelling along the $x$-arrow path which starts at $O_1$ we 
first encounter an arrow which vanishes at $E_e$ immediately after 
the vertex $P$. Arguing as in the proof of Proposition 
$\ref{prps-(3,3)-source-graph-to-case-1}$
we see that on one hand $X$ must lie on the $x$-arrow path $O_1 I_1 P$, 
while on the other hand it must lie somewhere within the $x$-tiled 
region or its boundary. On Figure $\ref{figure-16a}$ we see
that it is only possible if $X = O_1$ or $P$. We can't have $X =
O_1$ as then $i' = j' = k' = 0$, so $X = P$
and $i' = a + a_1$, $j' = b_2 = (- b_1) \mod b$ and $k' = c$. 
Assume now that $b \mid b_1$. Then $P = I_2$ and one can 
see on Figure $\ref{figure-16a}$ that the maximal path of 
$x$-arrows which starts at $O_1$ and in which no $x$-arrow
vanishes at $E_e$ is $O_1 I_1 I_2 O_2$. Arguing again as in the proof
of Proposition $\ref{prps-(3,3)-source-graph-to-case-1}$ we 
see that $X$ must belong both to $O_1 I_1 I_2 O_2$ and 
to the $x$-tiled region or its boundary. On Figure $\ref{figure-16a}$
we see that it is only possible if $X = O_1$, $I_2$ or $O_2$.
We can't have $X = O_1$ or $X = I_2$ as that would contradict
$i', j', k' \neq 0$. Therefore $X = O_2$ and the ratio is 
$x^{a + a_1 + a} : y^{b} z^{c}$. But this is impossible,
since $\frac{x^{a + a_1 + a}}{y^{b} z^{c}}$ decomposes as 
$\frac{x^{a + a_1}}{z^c}\; \frac{x^{a}}{y^{b}}$ and so can't
be carving out an edge of $\Sigma$ by 
Corollary $\ref{cor-primitivity-of-marking-ratios}$.
Similar argument for ratios of form $y^{j'} : x^{i'} z^{k'}$ shows 
that the only possibility is 
$y^{b + b_1} : x^{(- a_1) \mod a} z^{c}$ when $a \nmid a_1$. 

Suppose there is an edge incident to $e$ which is carved out by
a ratio $z^{k'} : x^{i'} y^{j'}$ for some $i',j',k' \neq 0$. Denote by $X$ 
the vertex corresponding to the common character of $z^{k'}$ and 
$x^{i'} y^{j'}$. As before, we see that on one hand $X$ must lie somewhere on
$z$-$(0,1)$-charge line $O_1 I_2$ and on the other hand it must lie
somewhere within the $z$-tiled region or its boundary. This is clearly only 
possible when $X = I_2$. Then $i' = c$, but due to non-contractibility 
of the $z$-tiled region we can no longer uniquely determine $j'$ and
$k'$. 

Suppose there is an edge incident to $e$ which is carved out by 
a ratio $x^{i'} : z^{k'}$ with $i',k' \neq 0$. Denote by $X$ the vertice
of $Q(G)$ corresponding to the common character of $x^{i'}$ and
$z^{k'}$. Then $X$ has to lie on both the $x$-arrow path $O_1 I_1 P$ and 
the $z$-arrow path $O_1 I_2$. 
From Figure $\ref{figure-16a}$ we see that this is only possible
when $X = P = I_2$, i.e. when $b \mid b_1$ and so $P$ coincides
with $I_2$. The ratio would then be $x^{a + a_1} : z^{c}$. A similar 
argument for ratios of form $y^{j'} : z^{k'}$ yields that we'd have 
to have $a \mid a_1$ and the ratio would have to be $y^{b+b1} : z^{c}$. 

Finally, suppose there is an edge incident to $e$ which is carved out
by a ratio $x^{i'} : y^{j'}$. Denote by $X$ the vertice
of $Q(G)$ corresponding to the common character of $x^{i'}$ and
$y^{j'}$. As before, we see $X$ would have to lie both on the
$x$-arrow path $O_1 I_1 P$ and on the $y$-arrow path $O_1 I_1 Q$. 
One possibility is always $X = I_1$, which yields the ratio 
$x^{a} : y^{b}$. As established above, any other intersection point
of $O_1 I_1 P$ and $O_1 I_1 Q$ would give rise to ratios of form
$x^{n'a} : y^{n'b}$ for some $n' \geq 2$. As the ratio marking
an edge in $\Sigma$ has to be primitive we conclude that the only 
possibility is $x^{a} : y^{b}$. 

Suppose $a \mid a_1$ and $b \mid a_2$. Then from the above
we see that the only ratios which could mark an edge incident
to $e$ would be:
\begin{align*} 
x^a : y^b, \quad y^{b + b_1} : z^c,  \quad z^c : x^{a+a_1}, \quad
z^c : x^{i'} y^{j'}.
\end{align*} 
for some $i',j' > 0$. 
Consulting the classification of the Proposition 
$\ref{prps-the-classification-of-the-vertices-of-the-triangulation}$ 
we see that $e$ must necessarily belong to Case
$\ref{case-one-line-from-vertex}(a)$, reproduced on Figure
$\ref{figure-16b}$, and we must have $i = a$, $j = b$,
$k_1=k_2=c$, $l = b + b_1$ and $m = a + a_1$.

Suppose $a \nmid a_1$ and $b \mid a_2$. Then the only ratios which
could mark an edge incident to $e$ would be: 
\begin{align*} 
x^a : y^b, \quad y^{b + b_1} : z^c x^{(-a_1)\mod a},  \quad 
x^{a+a_1} : z^c, \quad z^c : x^{i'} y^{j'}
\end{align*} 
for some $i',j' > 0$. 
Consulting the classification of the Proposition 
$\ref{prps-the-classification-of-the-vertices-of-the-triangulation}$
we see that $e$ must necessarily belong to Case
$\ref{case-one-line-from-vertex}(b)$, reproduced on Figure
$\ref{figure-16c}$ and we must have $i = a$, $j = b$,
$k_1=k_2=k_3=c$, $l = b + b_1$, $m = a + a_1$, $n = (-a_1) \mod a$. 
To compute $p$ and $q$ we use the following method. 
By construction of $\Sigma$ the 
lines carved out by the ratios $x^i : y^j$, $y^l : z^{k_3} x^{n}$ 
and $z^{k_1}:y^p x^q$ are parallel to sides of some regular triangle
and therefore themselves form a (degenerate) regular triangle (see
\cite{Craw-AnexplicitconstructionoftheMcKaycorrespondenceforAHilbC3}, 
Section 2 and \cite{Craw02}, Section 1.2). Three lines in $\Sigma$ are 
said to form a regular triangle if the product of ratios carving them out
(for some choice of one of the two mutually inverse Laurent monomials 
corresponding to each ratio) is $(x y z)^{r'}$ for some $r' \geq 0$. 
Such three lines must 
intersect at a triangle which has $r' + 1$ lattice points in each side,
with the degenerate case $r' = 0$ corresponding to the intersection 
being a point. Therefore in our case we must have 
\begin{align*}
\frac{x^i}{y^j} \; \frac{y^l}{z^{k3} x^{n}} \frac{z^{k1}}{y^p x^q} = 1.
\end{align*}
It follows that $p = l - j = b_1$ and $q = i - n = (a_1) \mod a$. 

The case $a \mid a_1$ and $b \nmid a_2$ is entirely analogus
to that of $a \nmid a_1$ and $b \mid a_2$.

Suppose that $a \nmid a_1$ and $b \nmid a_2$. Then the only ratios 
which could mark an edge incident to $e$ would be:
$$ x^a : y^b, \quad y^{b + b_1} : z^c x^{(-a_1)\mod a},  \quad 
x^{a+a_1} : z^c y^{(-b_1) \mod b}, \quad z^c : x^{i'} y^{j'}$$
for some $i',j' > 0$. Consulting all the possibilities for 
$e$ in the classification of the Proposition 
$\ref{prps-the-classification-of-the-vertices-of-the-triangulation}$
we see that $e$ could belong to either Case 
$\ref{case-one-line-from-vertex}(c)$ or Case 
$\ref{case-three-straight-lines}$. But were $e$ to belong to 
Case $\ref{case-three-straight-lines}$ it would have to be the
intersection point of three straight lines carved out by ratios
$x^{a+a_1} : z^c y^{(-b_1) \mod b}$, 
$y^{b + b_1} : z^c x^{(-a_1)\mod a}$ 
and $z^c : x^{r} y^{s}$ (see Figure $\ref{figure-05}$). 
These lines, by construction of $\Sigma$, 
form a regular triangle of side $0$ and therefore we would have 
to have: 
$$\frac{x^{a+a_1}}{z^c y^{(-b_1) \mod b}}
\frac{y^{b + b_1}}{z^c x^{(-a_1)\mod a}}
\frac{z^c}{x^{r} y^{s}} = 1.$$
This is impossible as the power of $z$ in the expression on LHS 
is clearly $-c$. We conclude that $e$ has to belong to Case
$\ref{case-one-line-from-vertex}(b)$ of the classification of the
Proposition 
$\ref{prps-the-classification-of-the-vertices-of-the-triangulation}$
and we must have  
\begin{align*}
&i = a, \quad j = b, \quad k_1=k_2=k_3=k4=c, \\
&l = b + b_1, \quad m = a + a_1, \quad n = (-a_1) \mod a, 
\quad r = (-b_1) \mod b.
\end{align*}
Computing $p$, $q$, $s$ and $t$ in the same way $p$ and $q$
were computed in the case $a \nmid a_1$ and $b \mid a_2$, we 
obtain:
$$p = b_1, \quad q = (a_1) \mod a, \quad s = (b_1) \mod b, \quad t =
a_1. $$

Finally, to obtain $e = \frac{1}{|G|}(bc, ac, |G| - bc - ac)$ 
we apply Lemma \ref{lemma-number-of-diamonds}. 
\end{proof}

\begin{prps}\label{prps-3-(1,2)-sources-to-case-3}
Let $e \in \mathfrak{E}$ be such that $E_e \subset \Except(Y)$. If 
the graph $SS_{\mathcal{M}, E_e}$ is as depicted on Figure 
\ref{figure-17a} then the vertex $e$ in the triangulation $\Sigma$ 
looks locally as depicted on Figure $\ref{figure-17b}$. 
The monomial ratios carving out the edges incident to $e$
can be computed in terms of the indicated lengths in 
$SS_{\mathcal{M}, E_e}$ as shown on Figure $\ref{figure-17b}$. 
The coordinates of $e$ in $L$ are $\frac{1}{|G|}(b c_3 + b_2 c -
b_2 c_3 , a c_2 + a_3 c - a_3 c_2, a b_3 + a_2 b - a_2 b_3)$. 
\end{prps}
\begin{figure}[h]
\centering 
\subfigure[Three $(2,1)$-source graph] { \label{figure-17a} 
\includegraphics[scale=0.12]{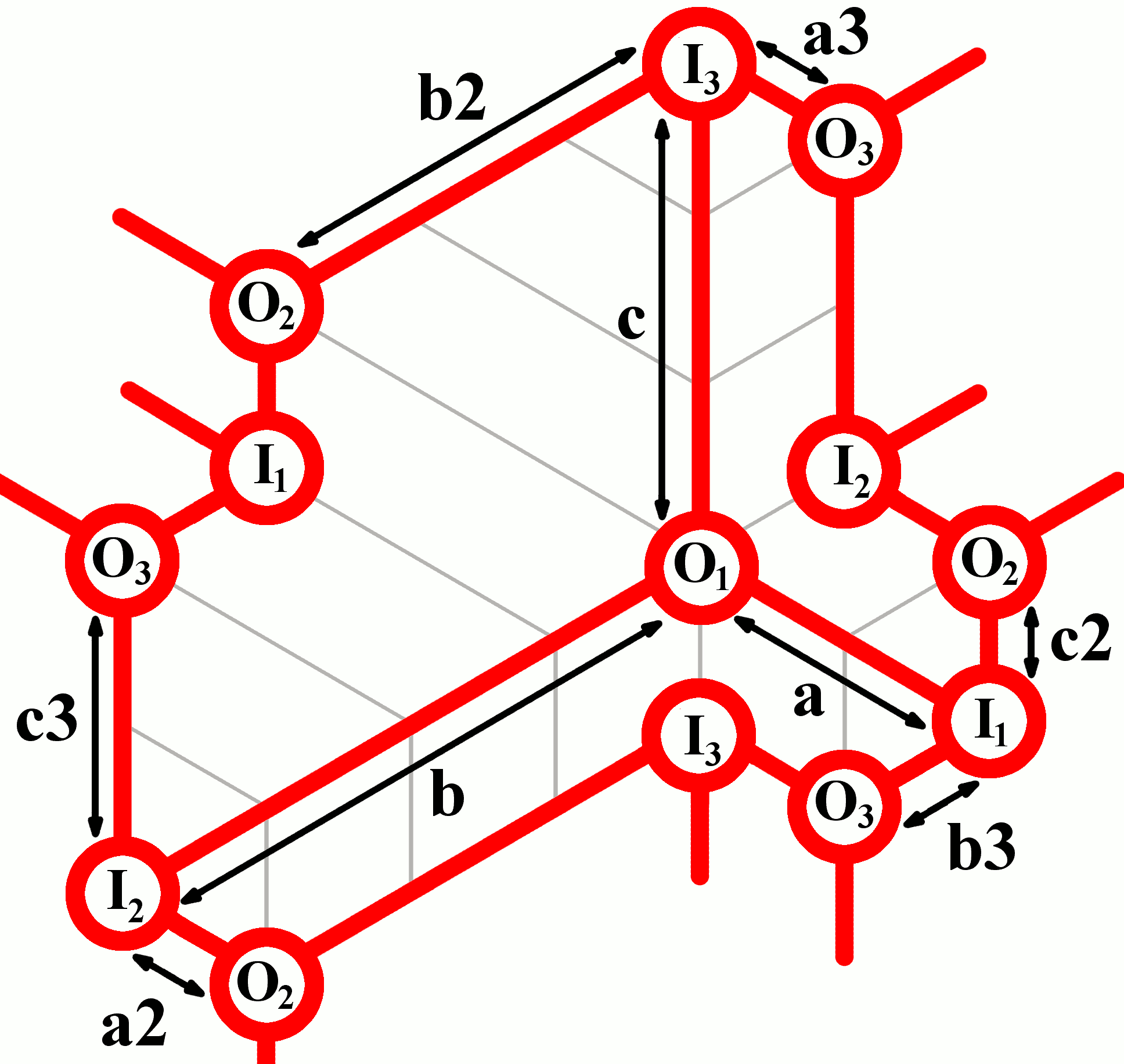}} \hspace{0.1cm}
\subfigure[Case $\ref{case-three-straight-lines}$ of Prps.
$\ref{prps-the-classification-of-the-vertices-of-the-triangulation}$] 
{\label{figure-17b} 
\raisebox{1.5cm}{\includegraphics[scale=0.35]{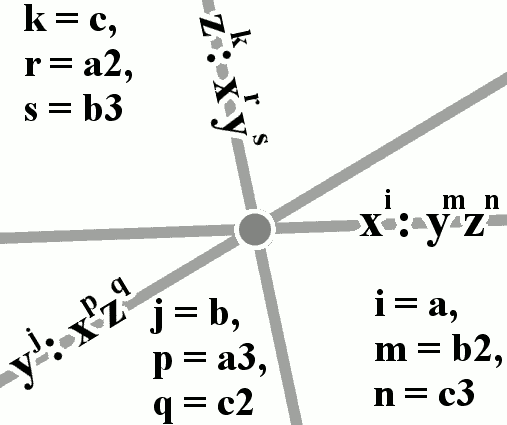}}}
\caption{The correspondence of Proposition
\ref{prps-3-(1,2)-sources-to-case-3}} \label{figure-17}
\end{figure}
\begin{proof}
Suppose there is an edge incident to 
$e$ which is carved out by ratio $x^{i'} : y^{j'} z^{k'}$ for some 
$i',j',k' \neq 0$. Denote by $X$ the vertex of $Q(G)$
corresponding to the common character of $x^{i'}$ and $y^{j'} z^{k'}$.
Arguing as in the proof of Proposition 
$\ref{prps-(3,3)-source-graph-to-case-1}$ 
we see that on one hand $X$ must lie on the $x$-$(0,1)$-charge line 
$O_1 I_1$, while on the other hand it must lie somewhere within 
the $x$-tiled region or its boundary. On Figure $\ref{figure-17a}$ we see
that it is only possible if $X = I_1$ and that the ratio would then 
have to be $x^{a} : y^{b_2} z^{c_3}$. Arguing similarly for ratios
of form $y^{j'} : x^{i'} z^{k'}$ and $z^{k'} : x^{i'} y^{j'}$ we
see that the only possibilities there are $y^{b} : x^{a_3} z^{c_2}$
and $z^{c} : x^{a_2} y^{b_3}$.

Suppose there is an edge incident to $e$ which is carved out 
by ratio $x^{i'} : y^{j'}$ for some $i',j' \neq 0$. Denote by $X$
the vertex corresponding to the common character of $x^{i'}$
and $y^{j'}$. Then $X$ must lie both on the $x$-$(0,1)$-charge
line $O_1 I_1$ and the $y$-$(0,1)$-charge line $O_1 I_2$. 
On Figure $\ref{figure-17a}$ we see that it is impossible 
unless $X = O_1$, but that would contradict $i', j' \neq 0$. 
Arguing similarly for ratios of form $y^{j'} : z^{k'}$ and
$z^{k'} : x^{i'}$ we see that they are impossible also. 

We conclude that the only ratios which could mark an edge incident
to $e$ would be:
\begin{align*} 
x^{a} : y^{b_2} z^{c_3}, \quad y^{b} : x^{a_3} z^{c_2}, 
z^{c} : x^{a_2} y^{b_3}
\end{align*}
Consulting the classification of the Proposition
$\ref{prps-the-classification-of-the-vertices-of-the-triangulation}$
we see that $e$ must necessarily belong to Case
$\ref{case-three-straight-lines}$, reproduced on Figure
$\ref{figure-17b}$, and that we must have
\begin{align*}
i = a, m = b_2, n = c_3, 
j = b, p = a_3, q = c_2,
k = c, r = a_2, s = b_3
\end{align*}

For $e = \frac{1}{|G|}(b c_3 + b_2 c - b_2 c_3 , a c_2 + a_3 c - a_3 c_2,
a b_3 + a_2 b - a_2 b_3)$ we apply Lemma \ref{lemma-number-of-diamonds}. 
\end{proof}

Since the sink source graph $SS_{\mathcal{M},E}$ of every exceptional
divisor $E \subset \Except(Y)$ is as depicted on either Figure
$\ref{figure-15a}$, Figure $\ref{figure-16a}$ or Figure $\ref{figure-17a}$, 
Propositions $\ref{prps-(3,3)-source-graph-to-case-1}$,
$\ref{prps-one-(2,1)-source-one-(1,2)-source-graph-to-case-2}$
and $\ref{prps-3-(1,2)-sources-to-case-3}$ add together to give
a following theorem:
\begin{theorem} \label{theorem-ssgraphs-to-divisor-types}
Let $e \in \mathfrak{E}$ be such that $E_e \subset \Except(Y)$. 
Then:
\begin{enumerate}
\item The graph $SS_{\mathcal{M},E_e}$ looks as on $\ref{figure-15a}$
if and only if the vertex $e$ in $\Sigma$ belongs to Case 
$\ref{case-meeting-of-champions}$ of the classification in Proposition
$\ref{prps-the-classification-of-the-vertices-of-the-triangulation}$. 
\item The graph $SS_{\mathcal{M},E_e}$ looks as on $\ref{figure-16a}$
if and only if the vertex $e$ in $\Sigma$ belongs to Case
$\ref{case-one-line-from-vertex}$ of the classification in
Proposition 
$\ref{prps-the-classification-of-the-vertices-of-the-triangulation}$. 
\item The graph $SS_{\mathcal{M},E_e}$ looks as on $\ref{figure-17a}$
if and only if the vertex $e$ in $\Sigma$ belongs to Case
$\ref{case-three-straight-lines}$ of the classification in
Proposition 
$\ref{prps-the-classification-of-the-vertices-of-the-triangulation}$.
\end{enumerate} \end{theorem}

We are now ready to prove Theorem
$\ref{theorem-reids-recipe-knows-everything}$. 
In the course of the proof we repeatedly use the fact that
for every $\chi \in G^\vee$ the transform $\Psi(\mathcal{O}_0 \otimes
\chi)$ is a shift of a coherent sheaf, that is - a complex all of 
whose cohomology sheaves are zero except for one 
(\cite{CautisLogvinenko-ADerivedApproachToGeometricMcKayCorrespondence}, 
Theorem $1.1$). If it is $k$-th
cohomology sheaf of $\Psi(\mathcal{O}_0 \otimes \chi)$ that 
doesn't vanish, we say that $\Psi(\mathcal{O}_0 \otimes \chi)$ is 
supported in degree $k$. 

We shall also need the following auxiliary results:

\begin{lemma}
\label{lemma-support-is-connected-and-consists-of-toric-orbit-closures}
For any $\chi \in G^\vee$ the support of $\Psi(\mathcal{O}_0 \otimes \chi)$ 
is connected and each of its irreducible components is either
a toric divisor $E_i$ or a toric curve $E_{i,j}$ with $i,j \in \mathfrak{E}$. 
\end{lemma}
\begin{proof}
By \cite{CautisLogvinenko-ADerivedApproachToGeometricMcKayCorrespondence}, 
Prop. $4.6$ we have it that $\Psi(\mathcal{O}_0 \otimes \chi)$ is the total 
complex of the skew-commutative cube of line bundles corressponding
to the subprepresentation $\hex(\chi^{-1})_{\tilde{\mathcal{M}}}$ of 
the associated representation $Q(G)_{\tilde{\mathcal{M}}}$. The
subquiver $\hex(\chi^{-1})$ is just the hexagonal subquiver
consisting of the six triangles in $Q(G)$ which contain $\chi^{-1}$. 
The support of $\Psi(\mathcal{O}_0 \otimes \chi)$ can then be computed 
with Lemma 3.1 of 
\cite{CautisLogvinenko-ADerivedApproachToGeometricMcKayCorrespondence}, 
which expresses it in terms of the vanishing divisors of the 
arrows of $\hex(\chi^{-1})_{\tilde{\mathcal{M}}}$. Recall now that for 
any arrow $q \in Q(G)$ its vanishing divisor $B_q$ in $\tilde{M}$ is 
of form $\sum_{f \in \mathfrak{E}}  b_{f} E_{f}$ with $b_{f} = 0,1$. With 
this in mind it follows from  
\cite{CautisLogvinenko-ADerivedApproachToGeometricMcKayCorrespondence},
Lemma 3.1, that every irreducible component of $\supp\left(\Psi(\mathcal{O}_0
\otimes \chi)\right)$ is an intersection of form $E_{f_1} \cap \dots
\cap E_{f_k}$ for $k \in \{1,2,3\}$ . Such
intersection is non-empty if and only if the cone $\sigma = \left<
f_1, \dots, f_k \right>$ is in the fan $\mathfrak{F}$ of $Y$. In which 
case it is precisely the toric orbit closure $E_\sigma$. We conclude
that each irreducible component of $\supp\left(\Psi(\mathcal{O}_0
\otimes \chi)\right)$ is either a toric divisor $E_{i}$, a toric curve
$E_{i,j}$ or a toric fixed point $E_{i,j,k}$ with $i,j,k \in
\mathfrak{E}$.

On the other hand, since $\Psi$ is an equivalence of derived
categories
$$ \eend_{D(Y)} \Psi(\mathcal{O}_0 \otimes \chi) =
\eend_{D^G(\mathbb{C}^3)} \mathcal{O}_{0} \otimes \chi =
\mathbb{C}. $$
Hence $\supp\left(\Psi(\mathcal{O}_0
\otimes \chi)\right)$ is connected. It can't therefore have
a toric fixed point as an irreducible component 
unless it is the only component. Which is impossible, as
then $\Psi(\mathcal{O}_0 \otimes \chi)$ would be a shift of a point
sheaf, but we know that \cite{BKR01} equivalence $\Phi$, of which 
$\Psi$ is the inverse, sends every point sheaf on $Y$ to 
a $G$-cluster on $\mathbb{C}^3$. The claim now follows.
\end{proof}

\begin{lemma} \label{lemma-divisor-in-supp-h1-iff-two-marked-curves}
Let $e \in \mathfrak{E}$ be such that $E_e \subset \Except(Y)$. For
any $\chi \in G^\vee$ the divisor $E_e$ belongs to the support
of $H^{-1}\left(\Psi(\mathcal{O}_0 \otimes \chi)\right)$
if and only if $E_e$ contains two or more curves marked by $\chi$. 
\end{lemma}
\begin{proof}
By Proposition \ref{sinks-and-sources-iff-in-transforms-support}
the divisor $E_e$ belongs to $\supp H^{-1}\left(\Psi(\mathcal{O}_0
\otimes \chi)\right)$ if and only if $\chi$ is a source vertex 
in $SS_{M,E_e}$. Figures $\ref{figure-15}$-$\ref{figure-17}$ list
all possible shapes of $SS_{\mathcal{M},E_e}$ together with
the corresponding toric fans of $e$. By inspection of this data
we see that $\chi$ is a source vertex in $SS_{\mathcal{M},E_e}$ 
if and only $\chi$ marks two or more edges incident to $e$ in 
the toric fan. Each of these edges corresponds to a toric curve
contained in $E_e$ and so the claim follows. 
\end{proof}

\begin{lemma} \label{lemma-support-h1-doesn't-contain-toric-curves}
Each of the irreducible components of 
$\supp H^{-1}\left(\Psi(\mathcal{O}_0 \otimes \chi)\right)$
is a toric divisor $E_e$ for some $e \in \mathfrak{E}$.  
\end{lemma}
\begin{proof}
As in the proof of Lemma 
\ref{lemma-support-is-connected-and-consists-of-toric-orbit-closures}
we use Lemma $3.1(2)$ of 
\cite{CautisLogvinenko-ADerivedApproachToGeometricMcKayCorrespondence}
to compute the support of $H^{-1}\left(\Psi(\mathcal{O}_0 \otimes \chi)\right)$
in terms of the vanishing divisors of the arrows in
$\hex(\chi^{-1})_{\tilde{\mathcal{M}}}$. 
Translating from the language of $Q(G)_{\tilde{\mathcal{M}}}$ into
that of $Q(G)_{\mathcal{\mathcal{M}}}$ as seen in the proof of Lemma
$\ref{class-in-a-family-from-a-class-in-its-dual}$ we see that
the vanishing divisors of $\hex(\chi^{-1})_{\tilde{\mathcal{M}}}$ are
precisely the vanishing divisors of $\hex(\chi)_{\mathcal{M}}$.
On Figure $\ref{figure-18}$ we marked for each arrow of $\hex(\chi)$ 
the name for its vanishing divisor in $\hex(\chi)_{\mathcal{M}}$ as
translated into the language of Lemma $3.1$ of 
\cite{CautisLogvinenko-ADerivedApproachToGeometricMcKayCorrespondence}.

\begin{figure}[!h] \begin{center}
\includegraphics[scale=0.24]{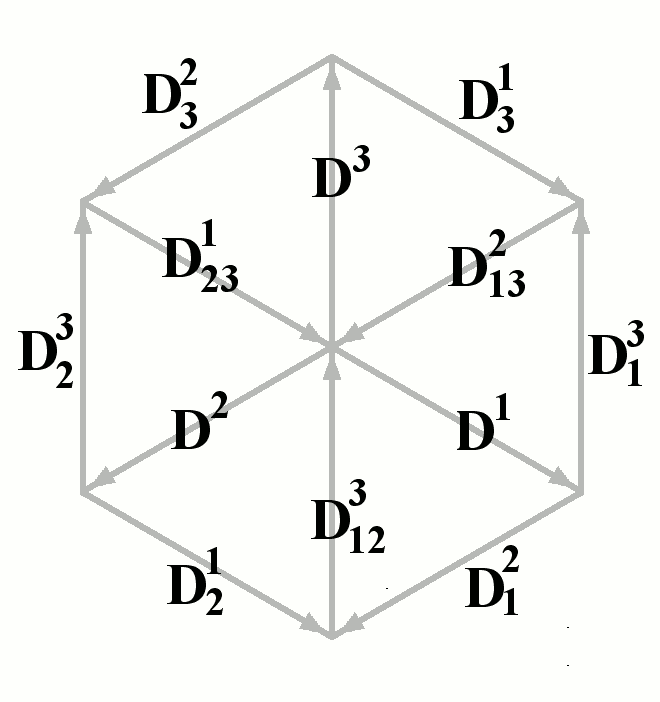} \end{center}
\caption{\label{figure-18}} The vanishing divisors of
$\hex(\chi)_{\mathcal{M}}$ in the language of Lemma $3.1$ of 
\cite{CautisLogvinenko-ADerivedApproachToGeometricMcKayCorrespondence}.
\end{figure}

By
\cite{CautisLogvinenko-ADerivedApproachToGeometricMcKayCorrespondence},
Lemma $3.1$,
$\supp H^{-1} \left( \Psi(\mathcal{O}_0
\otimes \chi)\right)$ has a filtration with successive quotients 
\begin{itemize} \item $\O_Z \otimes
\L_{12}(\gcd(D_1^2,D_2^1))$ where $Z$ is the scheme theoretic
intersection of $\gcd(D_1^2,D_2^1)$ and the effective part of $D^3 +
\lcm(\tD_3^1,\tD_3^2) - \tD_1^2 - D^1$ \item $\O_Z \otimes
\L_{13}(\gcd(D_1^3,D_3^1))$ where $Z$ is the scheme theoretic
intersection of $\gcd(D_1^3,D_3^1)$ and the effective part of $D^2 +
\lcm(D_2^1,\tD_2^3) - \tD_3^1 - D^3$ \item $\O_Z \otimes
\L_{23}(\gcd(D_2^3,D_3^2))$ where $Z$ is the scheme theoretic
intersection of $\gcd(D_2^3,D_3^2)$ and the effective part of $D^1 +
\lcm(D_1^2,D_1^3) - \tD_2^3 - D^2$ \end{itemize} where $\tD^i_j =
D^i_j - \gcd(D^i_j,D^j_i)$. 

Suppose $\supp H^{-1} \left(\Psi(\mathcal{O}_0 \otimes
\chi)\right)$ has an irreducible component which is less than a divisor 
and let us assume that it belongs to the third of 
the quotients in the filtration above. The other two cases are similar
and simpler. As each of $D^i_j$ is a sum 
$\sum_{e \in \mathfrak{E}} b_e E_e$ with 
$b_e = 0$ or $1$ we see that this component must be of form
$E_{e,f} = E_e \cap E_f$ with $E_e$ belonging to $\gcd(D_3^2,D_2^3)$ and 
$E_f$ to $D^1 + \lcm(D_2^1,D_3^1) - \tD_3^2 - D^2$. 
Every possible arrangement of which arrows of $\hex(\chi)$ vanish
along $E_e$ 
and which don't is listed in Figures \ref{figure-09} - \ref{figure-13}. 
Observe that $\chi$ can not be a source for $E_e$ as then 
$\supp H^{-1} \left(\Psi(\mathcal{O}_0 \otimes \chi)\right)$
would contain all of $E_e$ (Proposition
$\ref{sinks-and-sources-iff-in-transforms-support}$) and $E_{e,f}$ 
wouldn't be an irreducible component. Similarly $\chi$ can not be 
a sink as then $E_e$ 
would belong to either the support of $\supp H^{0}
\left(\Psi(\mathcal{O}_0 \otimes \chi)\right)$ or $\supp H^{-2}
\left(\Psi(\mathcal{O}_0 \otimes \chi)\right)$ 
and $\Psi(\mathcal{O}_0 \otimes \chi)$ 
wouldn't be supported in a single degree.  
By inspection we see that the only remaining 
possibility is that of $\chi$ being an $x$-$(0,1)$-charge for $E_e$. 
Arguing similarly for $E_f$ we see that $\chi$ must be either an $x$-tile,
or an $y$-$(1,0)$-charge or a $z$-$(1,0)$-charge. But observe that in each 
of these cases $D^2_{13}$ and $D^3_{12}$ contain $E_e$, while
$D^1_{23}$ contains $E_f$, so $D^1_{23} \cap D^2_{13} \cap D^3_{12}$
contains $E_{e,f}$. Then by Lemma $3.1(3)$ of
\cite{CautisLogvinenko-ADerivedApproachToGeometricMcKayCorrespondence}
$E_{e,f}$ would belong 
$\supp H^{-2} \left(\Psi(\mathcal{O}_0 \otimes \chi)\right)$,
and again $\Psi(\mathcal{O}_0 \otimes \chi)$ wouldn't be supported 
in a single degree. Therefore all irreducible components of
$\supp H^{-1} \left(\Psi(\mathcal{O}_0 \otimes \chi)\right)$ 
are divisors. The claim now follows from Lemma 
\ref{lemma-support-is-connected-and-consists-of-toric-orbit-closures}. 
\end{proof}

\begin{proof}[Proof of Theorem $\ref{theorem-reids-recipe-knows-everything}$]
We proceed case by case:

Proof of $(\ref{item-chi-marks-nothing})$: 

A character $\chi$ of $G$ marks nothing in Reid's recipe if and only
if it is the trivial character $\chi_0$. Using Lemma $3.1(3)$ of 
\cite{CautisLogvinenko-ADerivedApproachToGeometricMcKayCorrespondence}
as detailed in Lemma \ref{lemma-support-h1-doesn't-contain-toric-curves}
to compute $H^{-2} \left(\Psi(\mathcal{O}_0 \otimes \chi)\right)$
we see that for any $\chi \in G^\vee$
$$
H^{-2} \left(\Psi(\mathcal{O}_0 \otimes \chi)\right) = 
\mathcal{O}_Y(D) \otimes \mathcal{O}_D 
$$
where $D$ is the union of all $E \in \Except(Y)$ such 
that $\chi$ is a $(0,3)$-sink in $SS_{\mathcal{M}, E}$. 
By Proposition $4.14$ of
\cite{CautisLogvinenko-ADerivedApproachToGeometricMcKayCorrespondence}
the sink-source graph $SS_{\mathcal{M}, E}$ of any $E \in \Except(Y)$ 
has only one $(0,3)$-sink - the vertex $\chi_0$. Thus $H^{-2} 
\left(\Psi(\mathcal{O}_0 \otimes \chi)\right)$ is
$\mathcal{O}_Y(\Except(Y))
\otimes \mathcal{O}_{\Except(Y)}$ if $\chi =\chi_0$ and 
zero if $\chi \neq \chi_0$. The claim now follows. 
 
Proof of $(\ref{item-chi-marks-a-divisor})$:

\em The `If' direction: \rm Let $e \in \mathfrak{E}$ be such that 
$E_e \subset \Except(Y)$. We claim that it follows from Propositions 
$\ref{prps-(3,3)-source-graph-to-case-1}$, $\ref{prps-one-(2,1)-source-one-(1,2)-source-graph-to-case-2}$ and $\ref{prps-3-(1,2)-sources-to-case-3}$
that $\chi$ marks $E_e$ if and only if the vertex $\chi$
is a $(3,0)$-sink. Indeed, suppose the sink-source graph
$SS_{\mathcal{M}, E_e}$ is as depicted on Figure $\ref{figure-15a}$.
Then $SS_{\mathcal{M}, E_e}$ has a single $(3,0)$-sink $O_2$. 
Since $O_1$, the $(0,3)$-sink, is the trivial character $\chi_0$
we can see that $O_2$ is the character $\kappa(x^a y^b)$. By
Proposition $\ref{prps-(3,3)-source-graph-to-case-1}$ the 
triangulation $\Sigma$ around $e$ looks as on Figure
$\ref{figure-15b}$. Reid's recipe prescribes then for $E_e$
to be marked by $\xi^2$ where $\xi = \kappa(x^i) = \kappa (y^j) =
\kappa(z^k)$. Since by Proposition
$\ref{prps-(3,3)-source-graph-to-case-1}$ we have $i = a$, $j = b$
and $k = c$ we see that $O_2 = \kappa(x^a y^b) = \xi^2$ and 
the claim follows. The cases of $SS_{\mathcal{M}, E_e}$
being as on Figure $\ref{figure-16a}$ or on Figure $\ref{figure-17a}$
are treated similarly. After we express the powers of 
$x$,$y$ and $z$ in the ratios marking the edges incident to $e$
in terms of the lengths of the marked edges in $SS_{\mathcal{M},
E_e}$ the corresponding formula of Reid's recipe becomes the
natural formula for calculating $(3,0)$-sinks. 

On the other hand, by
Proposition $\ref{sinks-and-sources-iff-in-transforms-support}$ the
vertex $\chi$ is a $(3,0)$-sink for $E_e$ if and only if $E_e \subset
\supp H^0(\Psi(\mathcal{O}_0 \otimes \chi))$. We conclude
that $\chi$ marks $E_e$ if and only if $E_e \subset
\supp H^0(\Psi(\mathcal{O}_0 \otimes \chi))$. In particular,
if $\supp H^0\left(\Psi(\mathcal{O}_0 \otimes \chi)\right) = E_e$
then $\chi$ marks $E_e$. 

\em The `Only if' direction: \rm This was proved in 
Proposition 9.3 of \cite{Craw-Ishii-02}. There it
was done by showing that if $\chi$ marks some divisor $E$ then $\chi$
defines a wall of the $G$-$\hilb(\mathbb{C}^3)$ chamber in the space
of stability conditions for $G$-constellations and $E$ is the unstable
locus corresponding to crossing that wall. It then follows that
$\Psi(\mathcal{O}_0 \otimes \chi) = \mathcal{L}_{\chi}^{-1} \otimes
\mathcal{O}_E$. 

Proof of $(\ref{item-chi-marks-a-single-curve})$: 

By Lemma $3.1(1)$ of 
\cite{CautisLogvinenko-ADerivedApproachToGeometricMcKayCorrespondence}
we have:
$$H^{0} \left(\Psi(\mathcal{O}_0 \otimes \chi)\right) =
\mathcal{L}^{-1}_\chi \otimes \mathcal{O}_D.$$
where $D$ is the support of 
$H^{0} \left(\Psi(\mathcal{O}_0 \otimes \chi)\right)$. 
It suffices therefore to prove that $\chi$ marks 
a single toric curve in Reid's 
recipe if and only if this curve is the whole of
$\supp H^0\left(\Psi(\mathcal{O}_0 \otimes \chi)\right)$.   

Let $\left<e,f\right>$ be any two-dimensional cone in $\mathfrak{F}$ and 
let $E_{e,f}$ be the corresponding toric curve. 

\em Claim A: \rm If $\chi$ marks $E_{e,f}$ in Reid's recipe,
and $\supp H^0\left(\Psi(\mathcal{O}_0 \otimes \chi)\right)$
is non-empty then $E_{e,f}$ is an irreducible component of 
$\supp H^0\left(\Psi(\mathcal{O}_0 \otimes \chi)\right)$. 

\em Claim B: \rm If $E_{e,f}$ is an irreducible component
of $\supp H^0\left(\Psi(\mathcal{O}_0 \otimes \chi)\right)$ then
$E_{e,f}$ is the whole of 
$\supp H^0\left(\Psi(\mathcal{O}_0 \otimes \chi)\right)$ and 
$\chi$ marks $E_{e,f}$ in Reid's recipe. 

Suppose these claims were true and suppose in Reid's recipe 
$\chi$ marks $E_{e,f}$ and no other curve. Then $\chi$ is not the trivial 
character $\chi_0$ and so, as seen in the proof of
(\ref{item-chi-marks-nothing}), 
$\Psi(\mathcal{O}_0 \otimes \chi)$ is not supported in degree $-2$. 
Neither it is supported in degree $-1$, as we have it from Lemmas 
\ref{lemma-divisor-in-supp-h1-iff-two-marked-curves} and
\ref{lemma-support-h1-doesn't-contain-toric-curves} that 
the support of $H^{-1}\left(\Psi(\mathcal{O}_0 \otimes \chi)\right)$ 
consists of divisors which each contain two or more curves marked by
$\chi$. Therefore $\Psi(\mathcal{O}_0 \otimes \chi)$ is
supported in degree $0$ and we have 
it from Claim A that $E_{e,f}$ is an irreducible component of $\supp
H^0\left(\Psi(\mathcal{O}_0 \otimes \chi)\right)$ and then 
from Claim B that $E_{e,f}$ is the whole of 
$\supp H^0\left(\Psi(\mathcal{O}_0 \otimes \chi)\right)$.
Conversely suppose 
$E_{e,f} = \supp H^0\left(\Psi(\mathcal{O}_0 \otimes \chi)\right)$. 
Then by claim $B$ the character $\chi$ marks 
$E_{e,f}$ and by claim $A$ no other curves can be marked by $\chi$. 
Thus to prove $(\ref{item-chi-marks-a-single-curve})$ it suffices 
to prove these two claims.

\em Proof of Claim A: \rm Suppose $\chi$ marks $E_{e,f}$ and $\supp
H^0\left(\Psi(\mathcal{O}_0 \otimes \chi)\right)$ is non-empty. 
Then $\supp H^1\left(\Psi(\mathcal{O}_0 \otimes \chi)\right)$ is
empty and so by Proposition
$\ref{sinks-and-sources-iff-in-transforms-support}$ the vertex
corresponding to $\chi$ can not be a source in $SS_{\mathcal{M}, E}$
for any $E \in \Except(Y)$. In particular, $\chi$ is not a source
in $SS_{\mathcal{M}, E_e}$.  Let us consider all the possibilities
listed on Figures
$\ref{figure-15}$-$\ref{figure-17}$ for 
$SS_{\mathcal{M},E_e}$ and the corresponding toric fans of $e$ and 
let us check when is it possible
for $\chi$ in $SS_{\mathcal{M},E_e}$ not to be a source while marking 
one of the edges incident to $e$. We see that it is only possible
if the toric fan of $e$ is as
depicted on Figure $\ref{figure-16c}$, $\ref{figure-16d}$ or
$\ref{figure-16e}$ up to a permutation of $x$, $y$ and $z$. 
Assume without loss of generality that this permutation is such that 
the unique straight line passing through $e$ is of form 
$x^{\bullet} : y^{\bullet}$, i.e. exactly as on Figures
$\ref{figure-16c}$-$\ref{figure-16e}$. We also see that 
$\chi$ must be one of the vertices denoted on Figure $\ref{figure-16a}$
as $P$ or $Q$. Assume without loss of generality that $\chi$ is the vertex 
$P$. Then $\chi$ is a $y$-$(1,0)$-charge in $SS_{\mathcal{M},E_e}$
and $(e,f)$ is carved out by a ratio of form $x^{\bullet} :
z^{\bullet} y^{\bullet}$. Observe further that one of the two 
triangles containing $(e,f)$ must also 
contain the edge carved out by a ratio of form
$z^{\bullet} : y^{\bullet} x^{\bullet}$. But if one edge 
of any triangle in $\Sigma$ is carved out by a ratio of 
form $x^{\bullet} \colon y^{\bullet} z^{\bullet}$ and another edge
by a ratio of form $z^{\bullet} \colon x^{\bullet} y^{\bullet}$, 
then its third edge must be carved out by a ratio of form 
$y^{\bullet} : x^{\bullet}z^{\bullet}$ 
(see \cite{Craw-AnexplicitconstructionoftheMcKaycorrespondenceforAHilbC3}, 
\S 2). We conclude that one of the edges incident to $f$
is carved out by a ratio of form $y^{\bullet} : x^{\bullet}z^{\bullet}$. 
But now apply the same argument to $E_f$. 
We see that $f$ must also be as depicted on Figure
$\ref{figure-16c}$, $\ref{figure-16d}$ or $\ref{figure-16e}$ up 
to a permutation of $x$, $y$ and $z$. The permutation 
can not be such that the straight line which passes through 
$f$ is carved out by a ratio of form $z^{\bullet} : y^{\bullet}$. This
is because $(e,f)$ is carved out by a ratio of form 
$x^{\bullet} : y^{\bullet} z^{\bullet}$ which under such permutation 
would correspond to $z^{\bullet} : x^{\bullet} y^{\bullet}$ on 
Figures $\ref{figure-16c}$-$\ref{figure-16e}$. Which is impossible
as then $\chi$, the character marking $(e,f)$, 
would be the vertex $I_2$ on Figure $\ref{figure-16a}$ which is a source 
in $SS_{\mathcal{M},E_f}$. But neither can the straight line passing 
through $f$ be carved out by a ratio
of form $x^{\bullet} : y^{\bullet}$ as then $f$ has only one edge 
carved out by the ratio of form $y^{\bullet} : x^{\bullet}z^{\bullet}$
and this edge clearly can not be contained in the same triangle
as the edge $(e,f)$ which is marked by a ratio of form $x^{\bullet}
\colon y^{\bullet} z^{\bullet}$ (see Figure \ref{figure-16e}). We conclude
that the straight line passing through $f$ must be carved out by a
ratio of form $x^{\bullet} : z^{\bullet}$.  Then, since $(e,f)$ is 
carved out by $x^{\bullet} : y^{\bullet} z^{\bullet}$, the 
vertex $\chi$ which marks $(e,f)$ has to be a 
$z$-$(1,0)$-charge in $SS_{\mathcal{M}, E_{f}}$. 
Since $\chi$ is also $y$-$(1,0)$-charge in $SS_{\mathcal{M},
E_{f}}$ we conclude by consulting Figure 
$\ref{figure-10}$) that each of the three arrows in $Q(G)$ whose 
tail is $\chi$ vanishes either along $E_e$ or along $E_f$ but
all three of them vanish neither along $E_e$ nor along $E_f$. 
This by Proposition $4.6$ and Lemma $3.1(1)$ of
\cite{CautisLogvinenko-ADerivedApproachToGeometricMcKayCorrespondence},
and translating from the language of $Q(G)_{\tilde{M}}$ into
that of $Q(G)_{\mathcal{M}}$ as seen in the proof of Lemma
$\ref{class-in-a-family-from-a-class-in-its-dual}$, implies
that $E_{e,f}$ belongs to the support of $H^0 \left(\Psi(\mathcal{O}_0 \otimes
\chi)\right)$, but neither $E_{e}$ nor $E_{f}$ do. Therefore
$E_{e,f}$ is an irreducible component of $\supp H^0
\left(\Psi(\mathcal{O}_0 \otimes \chi)\right)$. 

\em Proof of Claim B: \rm Denote by $D$ the support
of $H^0\left(\Psi\left(\mathcal{O}_0 \otimes \chi\right)\right)$
and suppose $E_{e,f}$ is an irreducible component of $D$. 
By Lemma \ref{lemma-support-is-connected-and-consists-of-toric-orbit-closures}
each irreducible component of $D$ is either a toric divisor 
or a toric curve. But as seen in the proof of the 
\em`if' \rm direction of (\ref{item-chi-marks-a-divisor})
$D$ contains a toric divisor $E$ if and only if $\chi$ marks $E$.  
And by the  \em `only if' \rm direction  
(\ref{item-chi-marks-a-divisor}) if $\chi$ marks $E$ then $E$
is the whole of $D$. We conclude that every irreducible component 
of $D$ is a toric curve. We wish to show that $E_{e,f}$
is the only such component and to do that we have to roll up 
our sleeves and calculate some sheaf cohomology.

Recall that $\Psi: D^G(\mathbb{C}^3) \rightarrow D(Y)$ was defined
as the inverse of the Fourier-Mukai equivalence $\Phi: D(Y) \rightarrow
D^G(\mathbb{C}^3)$ of \cite{BKR01}. For any $F \in D(Y)$ we can compute 
the global sections of
$\chi$-eigenparts of $\Phi(F)$ by taking a derived pushdown in 
two different ways. 
By definition 
\begin{align*}
\Phi(F) = \rder \pi_{\mathbb{C}^3 *} 
\left(\mathcal{M}
\overset{\lder}{\otimes} \pi^*_{Y} (F \otimes \chi_0) \right)
\end{align*}
where $\pi_{Y}$ and $\pi_{\mathbb{C}^3}$ are the projections
from $Y \times \mathbb{C}^3$ onto $Y$ and $\pi_{\mathbb{C}^3}$. 
Let $\gamma_{\mathbb{C}^3}:\; \mathbb{C}^3 \rightarrow \mathbb{C}^3/G$ 
and $\gamma_{Y}: Y \rightarrow \mathbb{C}^3/G$ be 
the quotient map and the resolution morphism. Making use of the projection 
formula, we have 
$$ \rder \gamma_{\mathbb{C}^3 *}  \Phi(F) = 
\rder \left(\pi_{\mathbb{C}^3} \circ \gamma_{\mathbb{C}^3}\right)_* 
\left(\mathcal{M}
\overset{\lder}{\otimes} \pi^*_{Y} (F \otimes \chi_0) \right) = $$
$$ = \rder \left(\pi_{Y} \circ \gamma_{Y}\right)_*
\left(\mathcal{M} \overset{\lder}{\otimes} \pi^*_{Y} 
(F \otimes \chi_0) \right) = $$
$$  = \rder \gamma_{Y *} \left(\bigoplus_{\chi \in G^\vee} \mathcal{L}_\chi
\otimes (F \otimes \chi_0))\right) $$
Let $F$ be a sheaf in $\cohcat(Y)$. Taking global sections and making use 
of $G$-equivariance, we see that for any $\xi \in G^\vee$
\begin{align} \label{eqn-chi-eigenparts-via-sheaf-cohomology}
H^i\; \Gamma(\Phi(F)_\xi) = H^i(\mathcal{L}_\xi \otimes F) 
\end{align} 
where on the RHS we take the $i$-th sheaf cohomology and on the LHS we take 
the $i$-th cohomology of the vector space complex $\Gamma(\Phi(F)_\xi)$
where the complex $\Phi(F)_\xi$ is the $\xi$-eigenpart of 
the complex $\Phi(F)$.

Let now $\chi'$ be any character of $G$. 
As $\Phi(\mathcal{L}^{-1}_\chi \otimes \mathcal{O}_D) = \mathcal{O}_0 
\otimes \chi$ 
setting $F = \mathcal{L}^{-1}_\chi \otimes \mathcal{O}_D$ and $\xi =
\chi'$ in $\eqref{eqn-chi-eigenparts-via-sheaf-cohomology}$ yields
that 
\begin{align*}
\chi \left( \mathcal{L}_{\chi'} \otimes \mathcal{L}^{-1}_{\chi}
\otimes \mathcal{O}_D \right) = 
\begin{cases}
0 \text{ for } \chi' \neq \chi \\
1 \text{ for } \chi' =  \chi
\end{cases}.
\end{align*}
By $\chi(-)$ we denote the Euler characteristic 
$\sum_{i \in \mathbb{Z}} (-1)^i \dim H^{i}(-)$.
Then for $\chi' \neq \chi$ we have 
\begin{align} \label{eqn-euler-characteristic-of-deg-0-support}
\notag 0 = \chi\left(\mathcal{L}_{\chi'} \otimes \mathcal{L}^{-1}_{\chi}
\otimes \mathcal{O}_D\right)  
& = \sum_{\sigma \in D} \deg_{E_{\sigma}} \left(\mathcal{L}_{\chi'}
\otimes \mathcal{L}^{-1}_{\chi}\right) + \chi(\mathcal{O}_D) = \\
& = \left( \sum_{\sigma \in D} \deg_{E_{\sigma}} \mathcal{L}_{\chi'} \right) - 
\left( \sum_{\sigma \in D} \deg_{E_{\sigma}} \mathcal{L}_\chi \right)
+ \chi(\mathcal{O}_D).
\end{align} 
where we abuse the notation 
by writing $\sigma \in D$ to mean that $\sigma$ is a two-dimensional
cone in $\mathfrak{F}$ such that $E_{\sigma} \subset D$. Observe that 
the sum $\sum_{\sigma \in D} \deg_{E_{\sigma}} \mathcal{L}_{\chi'}$ 
doesn't depend on the choice of $\chi' \neq \chi$. 
But $\chi_0 \neq \chi$ as otherwise $\Psi(\mathcal{O}_0 \otimes
\chi)$ would have to be supported in degree $-2$ by part
(\ref{item-chi-marks-nothing}) of this theorem. Evaluating  
$\sum_{\sigma \in D} \deg_{E_{\sigma}} \mathcal{L}_{\chi'}$
for $\chi' = \chi_0$ we obtain zero since
$\mathcal{L}_{\chi_0} = \mathcal{O}_Y$. Therefore  
$\sum_{\sigma \in D} \deg_{E_{\sigma}}
\mathcal{L}_{\chi'} = 0$ for any $\chi' \neq \chi$. On the other
hand by Lemma $\ref{lemma-markings-divide-nonequal-ggraph-pieces}$
the degree of $\mathcal{L}_{\chi'}$ is non-negative on any toric curve 
$E_{\sigma}$. And by Corollary \ref{cor-markings-belong-to-ggraphs} 
the degree of $\mathcal{L}_{\chi'}$ is $1$ on any curve marked by 
$\chi'$. Therefore for any $\chi'$ which marks any of the curves in $D$  
we have $\sum_{\sigma \in D} \deg_{E_{\sigma}}
\mathcal{L}_{\chi'} \geq 1$. We conclude that $\chi$ marks
all the curves in $D$. Assume that $D$ contains some
curve $E_{e',f'}$ other then $E_{e,f}$. As $D$ is connected $E_{e',f'}$ 
must intersect $E_{e,f}$. Then in $\Sigma$ the edges $(e,f)$ and
$(e',f')$ must be two sides of some regular triangle and therefore
have a common vertex. Without loss of generality assume $e = e'$. 
Then $E_e$ contains two curves marked by $\chi$ and by Lemma 
\ref{lemma-divisor-in-supp-h1-iff-two-marked-curves}
it must belong to $H^{-1} \supp \left(\Psi(\mathcal{O}_0 \otimes
\chi)\right)$, which 
contradicts $\Psi(\mathcal{O}_0 \otimes \chi)$ being supported in
a single degree. We conclude that $E_{e,f}$ is the whole of $D$ and
the claim follows.  

Proof of $(\ref{item-chi-marks-several-curves})$:

\em The `If' direction: \rm  Suppose $\Psi(\mathcal{O}_0 \otimes \chi)$
is supported in degree $-1$. Then $\chi$ can not mark a divisor
or mark a single curve or mark nothing in Reid's recipe since then 
$\Psi(\mathcal{O}_0 \otimes \chi)$ would be necessarily supported
in degree $0$ or degree $-2$ by
by parts $(\ref{item-chi-marks-a-divisor})$, $(
\ref{item-chi-marks-a-single-curve})$, $(\ref{item-chi-marks-nothing})$
of this theorem which we already proved.  
Hence $\chi$ must mark several curves in Reid's recipe.

\em The `Only if' direction: \rm 

Assume that $\chi$ marks several curves in Reid's recipe. Then $\chi
\neq \chi_0$ and by the proof of the part 
$(\ref{item-chi-marks-nothing})$ of this theorem 
$\Psi(\mathcal{O}_0 \otimes \chi)$
is not supported in degree $-2$. Nor can $\Psi(\mathcal{O}_0 \otimes \chi)$
be supported in degree $0$. For by Lemma
\ref{lemma-support-is-connected-and-consists-of-toric-orbit-closures}
the irreducible components of the support of $\Psi(\mathcal{O}_0 \otimes \chi)$
are toric divisors and toric curves. And as seen in the proof of the `if'
direction of the part $(\ref{item-chi-marks-a-divisor})$ 
of this theorem a divisor belongs to 
$\supp H^0\left(\Psi(\mathcal{O}_0 \otimes \chi)\right)$ 
if and only $\chi$ marks this divisor in Reid's recipe. 
Similarly, as we seen in the proof of part
$(\ref{item-chi-marks-a-single-curve})$, a toric curve is an
irreducible component of 
$\supp H^0\left(\Psi(\mathcal{O}_0 \otimes \chi)\right)$
if and only if $\chi$ marks just this curve in Reid's recipe.
We conclude that $\Psi(\mathcal{O}_0 \otimes \chi)$ is
supported in degree $-1$. 

Finally, the fact that $\supp H^{-1} \left(\Psi(\mathcal{O}_0 \otimes
\chi)\right)$ is precisely the union of the divisors containing
two or more curves marked by $\chi$ is a consequence of Lemmas 
\ref{lemma-divisor-in-supp-h1-iff-two-marked-curves} and
\ref{lemma-support-h1-doesn't-contain-toric-curves}. 

This concludes our proof of Theorem 
\ref{theorem-reids-recipe-knows-everything}. 
\end{proof}

\begin{proof}[Proof of Theorem $\ref{theorem-geometrical-mckay-correspondence}$]
We get this as a free consequence of Theorem 
\ref{theorem-reids-recipe-knows-everything}
by observing that any non-trivial $\chi \in G^\vee$ must either mark a divisor, 
a single curve or several curves (Corollary $4.6$ of
\cite{Craw-AnexplicitconstructionoftheMcKaycorrespondenceforAHilbC3}). 
Thus any $\chi \in G^\vee$ must belong to one of the four cases in
Theorem $\ref{theorem-reids-recipe-knows-everything}$.
\end{proof}

\bibliography{../../references}
\bibliographystyle{amsalpha}

\end{document}